\documentclass[article]{siamart0216} 

\usepackage{etoolbox}
\providetoggle{comments}\settoggle{comments}{true}

\allowdisplaybreaks
\usepackage{amssymb,amsfonts,amsopn,mathrsfs,calligra}
\usepackage{graphicx,epstopdf,subfig}
\usepackage{algpseudocode,booktabs,framed}
\usepackage{tikz,pgfplots}
\usepackage{hyperref}

\iftoggle{comments}
  {\usepackage[paperwidth=7.32in,paperheight=11.8in]{geometry}
    \setlength{\evensidemargin}{0.4in}
    \setlength{\oddsidemargin}{-0.4in}
    \usepackage{todonotes}
    \newcommand\ed[1]{\todo[color=black!10]{#1}}
    \newcommand\lc[1]{\todo[color=red!10]{#1}}
    \newcommand\edi[1]{\todo[color=black!10,inline]{#1}}
    \newcommand\lci[1]{\todo[color=red!10,inline]{#1}}
   }
   {\newcommand\ed[1]{} \newcommand\edi[1]{} \newcommand\lci[1]{} \newcommand\lc[1]{}}

\newcommand{\revrev}[1]{#1}

\DeclareFontFamily{U}{mathx}{\hyphenchar\font45}
\DeclareFontShape{U}{mathx}{m}{n}{
      <5> <6> <7> <8> <9> <10>
      <10.95> <12> <14.4> <17.28> <20.74> <24.88>
      mathx10
      }{}
\DeclareSymbolFont{mathx}{U}{mathx}{m}{n}
\DeclareFontSubstitution{U}{mathx}{m}{n}
\DeclareMathAccent{\widecheck}{0}{mathx}{"71}
\DeclareMathAccent{\wideparen}{0}{mathx}{"75}

\ifpdf
  \DeclareGraphicsExtensions{.eps,.pdf,.png,.jpg}
\else
  \DeclareGraphicsExtensions{.eps}
\fi

\newcommand{\TheTitle}{Fast Low-Rank Kernel Matrix Factorization using Skeletonized Interpolation}
\newcommand{\TheAuthors}{L.\ Cambier, E. Darve}

\headers{Skeletonized Interpolation}{\TheAuthors}

\title{{\TheTitle}}

\author{
    L{\'e}opold Cambier\thanks{Institute for Computational \& Mathematical Engineering, Stanford University. Huang Engineering Center, 475 Via Ortega, Suite B060, Stanford, CA-94305, USA. \email{lcambier@stanford.edu}, \url{https://stanford.edu/\~lcambier}}
  \and
Eric Darve\thanks{Department of Mechanical Engineering, Stanford University. 452 Escondido Mall, Bldg 520, Room 125, Stanford, CA-94305, USA. \email{darve@stanford.edu}}
}

\newcommand{\R}{\mathbb{R}}
\newcommand{\OO}[1]{\mathcal{O}\!\left(#1\right)}
\newcommand{\dif}{\text{d}}
\newcommand{\Xspace}{\mathcal{X}}
\newcommand{\Yspace}{\mathcal{Y}}
\newcommand{\Xcheb}{\overline X}
\newcommand{\Ycheb}{\overline Y}
\newcommand{\Xcheba}{\widehat X}
\newcommand{\Ycheba}{\widehat Y}
\newcommand{\Xchebb}{\widecheck X}
\newcommand{\Ychebb}{\widecheck Y}
\newcommand{\Xtilde}{\widetilde X}
\newcommand{\Ytilde}{\widetilde Y}
\newcommand{\Eint}{E_{\textrm{INT}}}
\newcommand{\Eqr}{E_{\textrm{QR}}}
\newcommand{\Kfun}{\mathscr{K}}
\newcommand{\KfunI}{\overline{\mathscr{K}}}
\newcommand{\afun}{a}
\newcommand{\ufun}{u}
\newcommand{\ffun}{f}

\newcommand{\nx}{{\overline{m}}}
\newcommand{\ny}{{\overline{n}}}
\newcommand{\Shat}{\widehat{S}}
\newcommand{\That}{\widehat{T}}
\newcommand{\Wx}{\overline{W}_X}
\newcommand{\Wy}{\overline{W}_Y}
\newcommand{\DWx}{\diag(\Wx)}
\newcommand{\DWy}{\diag(\Wy)}
\newcommand{\DWxa}{\diag(\widehat W_X)}
\newcommand{\DWya}{\diag(\widehat W_Y)}
\newcommand{\Kwhat}{\hat{K}_w}

\DeclareMathOperator{\diag}{diag}

\begin{document}

\maketitle

\begin{abstract} 
Integral equations are commonly encountered when solving complex physical problems. Their discretization leads to a dense kernel matrix that is block or hierarchically low-rank. This paper proposes a new way to build a low-rank factorization of those low-rank blocks at a nearly optimal cost of $\OO{n r}$ for a $n \times n$ block submatrix of rank $r$. This is done by first sampling the kernel function at new interpolation points, then selecting a subset of those using a CUR decomposition and finally using this reduced set of points as pivots for a RRLU-type factorization. We also explain how this implicitly builds an optimal interpolation basis for the Kernel under consideration. We show the asymptotic convergence of the algorithm, explain its stability and demonstrate on numerical examples that it performs very well in practice, allowing to obtain rank nearly equal to the optimal rank at a fraction of the cost of the naive algorithm.
\end{abstract}



\section{Introduction}

In this paper, we are interested in the low-rank approximation of kernel matrices, i.e., matrices $K_{ij}$ defined as
\[ K_{ij} = \Kfun(x_i, y_j) \]
for $x_i \in X = \{x_1, \dots, x_m\} \subseteq \Xspace$ and $y_j \in Y = \{y_1, \dots, y_n\} \subseteq \Yspace$ and where $\Kfun$ is a smooth function over $\Xspace \times \Yspace$.
A typical example is when
\[ \Kfun(x, y) = \frac 1{\|x - y\|_2} \]
and $X \subset \Xspace$ and $Y \subset \Yspace$ are two well-separated sets of points.

This kind of matrices arises naturally when considering integral equations like
\[ \afun(x) \ufun(x) + \int_{\tilde{\Yspace}} \Kfun(x, y) \ufun(y) \dif y = \ffun(x) \quad \forall x \in \tilde\Xspace \]
where the discretization leads to a linear system of the form
\begin{equation} \label{eq:ls} 
  a_i u_i + \sum_j K_{ij} u_j = f_i 
\end{equation}
where $K$ is a \emph{dense} matrix. While this linear system as a whole is usually not low-rank, one can select subsets of points $X \subset \Xspace$ and $Y \subset \Yspace$ such that $\Kfun$ is smooth over $\Xspace \times \Yspace$ and hence $\Kfun(X, Y)$ is low-rank. This corresponds to a submatrix of the complete $K$. Being able to efficiently compute a low-rank factorization of such submatrix would lead to significant computational savings. By ``smooth" we usually refer to a function with infinitely many continuous derivatives over its domain. Such a function can be well approximated by its interpolant at Chebyshev nodes for instance. 

Low-rank factorization means that we seek a factorization of $K = \Kfun(X, Y)$ as
\[ K = U S V^\top \]
where $U \in \R^{m \times r}, V \in \R^{n \times r}$, $S \in \R^{r \times r}$, and $r$ is the rank. In that factorization, $U$ and $V$ don't necessarily have to be orthogonal. One way to compute such a factorization is to first compute the matrix $K$ at a cost $\OO{m n}$ and then to perform some rank-revealing factorization like SVD, rank-revealing QR or rank-revealing LU at a cost usually proportional to $\OO{m n r}$. But, even though the resulting factorization has a storage cost of $\OO{(m + n) r}$, linear in the size of $X$ and $Y$, the cost would be proportional to $\OO{m n}$, i.e., quadratic.

\subsection{Notation}

In the following, we will denote by $\Kfun$ a function over $\Xspace \times \Yspace$. $X$ and $Y$ are finite sequences of vectors such that $X \subset \Xspace$ and $Y \subset \Yspace$ and $\Kfun(X, Y)$ denotes the matrix $K_{ij} = \Kfun(x_i, y_j)$. Small-case letters $x$ and $y$ denote arbitrary variables, while capital-case letters $\Xcheb$, $\Xcheba$, $\Xchebb$, $\Xtilde$ denotes sequences of vectors. We denote matrices like $A(X, Y)$ when the rows refer to the set $X$ and the columns to the set $Y$.
\autoref{tab:notation} summarizes all the symbols used in this paper.

\begin{table}
    \centering
    \begin{tabular}{lp{250pt}}
        \toprule
        $\Kfun$ & The smooth kernel function \\
        $\Xspace, \Yspace$ & The spaces over which $\Kfun$ is defined, i.e., $\Xspace \times \Yspace$ \\
        $x$, $y$ & Variables, $x \in \Xspace$, $y \in \Yspace$ \\
        $X$, $Y$ & The mesh of points over which to approximate $\Kfun$, i.e., $X \times Y$ \\
        $K$ & The kernel matrix, $K = \Kfun(X, Y)$, $K_{ij} = \Kfun(x_i, y_j)$ \\
        $m$, $n$ & $m = |X|$, $n = |Y|$ \\
        $\Xcheb$, $\Ycheb$ & The tensor grids of Chebyshev points \\
        $\Xcheba$, $\Ycheba$ & The subsets of $\Xcheb$ and $\Ycheb$ output by the algorithm used to build the low-rank approximation \\
        $\nx$, $\ny$ & The number of Chebyshev tensor nodes, $\nx = |\Xcheb|$, $\ny = |\Ycheb|$ \\
        $r_0$ & The ``interpolation" rank of $\Kfun$, i.e., $r_0 = \min(|\Xcheb|, |\Ycheb|)$ \\
        $r_1$ & The Skeletonized Interpolation rank of $\Kfun$, i.e.,  $r_1 = |\Xcheba| = |\Ycheba|$ \\
        $r$ & The rank of the continuous SVD of $\Kfun$ \\
        $S(x, \Xcheb)$, $T(y, \Ycheb)$ & Row vectors of the Lagrange basis functions, based on $\Xcheb$ and $\Ycheb$ and evaluated at $x$ and $y$, respectively. Each column is one Lagrange basis function. \\
        $\Shat(x, \Xcheba)$, $\That(y, \Ycheba)$ & Row vectors of Lagrange basis functions, based on $\Xcheba$ and $\Ycheba$, built using the Skeletonized Interpolation  and evaluated at $x$ and $y$, respectively. Each column is one function.\\
        $w_k$, $w_l$ & Chebyshev integration weights \\
        $\diag(\Wx)$,  $\diag(\Wy)$ & Diagonal matrices of integration weights when integration is done at nodes $\Xcheb$ and $\Ycheb$ \\
        \bottomrule
    \end{tabular}
    \caption{Notations used in the paper}
    \label{tab:notation}
\end{table}

\subsection{Previous work}

The problem of efficiently solving \eqref{eq:ls} has been extensively studied in the past. As indicated above, discretization often leads to a dense matrix $K_{ij}$. Hence, traditional techniques such as the LU factorization cannot be applied because of their $\OO{n^3}$ time and even $\OO{n^2}$ storage complexity.
The now traditional method used to deal with such matrices is to use the fact that they usually present a (hierarchically) low-rank structure, meaning we can represent the matrix as a hierarchy of low-rank blocks. The Fast Multipole Method (FMM) \cite{rokhlin1985rapid, greengard1987fast, barnes1986hierarchical} takes advantage of this fact to accelerate computations of matrix-vector products $Kv$ and one can then couple this with an iterative method.
More recently, \cite{fong2009black} proposed a kernel-independent FMM based on interpolation of the kernel function.

Other techniques compute explicit low-rank factorization of blocks of the kernel matrix through approximation of the kernel function.
The Panel Clustering method \cite{hackbusch1989fast} first computes a low-rank approximation of $\Kfun(x, y)$ as
\[ \Kfun(x, y) \approx \sum_i \kappa_i(x ; y_0) \phi_i(y) \]
by Taylor series and then uses it to build the low-rank factorization.

Bebendorf and Rjasanow proposed the Adaptive Cross Approximation \cite{bebendorf2003adaptive}, or ACA, as a technique to efficiently compute low-rank approximations of kernel matrices. ACA has the advantage of only requiring to evaluate rows or columns of the matrix and provides a simple yet very effective solution for smooth kernel matrix approximations. However, it can have convergence issues in some situations (see for instance \cite{borm2005hybrid}) if it cannot capture all necessary information to properly build the low-rank basis and lacks convergence guarantees.

In the realm of analytic approximations, \cite{yarvin1998generalized} (and similarly \cite{borm2004low}, \cite{borm2005hybrid}, \cite{fong2009black} and \cite{wu2014optimized} in the Fourier space) interpolate $\Kfun(x, y)$ over $\Xspace \times \Yspace$ using classical interpolation methods (for instance, polynomial interpolation at Chebyshev nodes in \cite{fong2009black}), resulting in expressions like
\[ \Kfun(x, y) \approx S(x, \Xtilde) \Kfun(\Xtilde, \Ytilde) T(y, \Ytilde)^\top = \sum_{k} \sum_{l} S_k(x) \Kfun(\tilde x_k, \tilde y_l) T_l(y) \]
where $S$ and $T$ are Lagrange interpolation basis functions.
Those expressions can be further recompressed by performing a rank-revealing factorization on the node matrix $\Kfun(\Xtilde, \Ytilde)$, for instance using SVD \cite{fong2009black} or ACA \cite{borm2005hybrid}.
Furthermore, \cite{yarvin1998generalized} takes the SVD of a scaled $\Kfun(\Xtilde, \Ytilde)$ matrix to further recompress the approximation and obtain an explicit expression for $u_r$ and $v_r$ such that
\[ \Kfun(x, y) \approx \sum_s \sigma_s u_s(x) v_s(y) \]
where $\{u_s\}_s$ and $\{v_s\}_s$ are sequences of orthonormal functions in the usual $L_2$ scalar product.

Bebendorf \cite{bebendorf2000approximation} builds a low-rank factorization of the form
\begin{equation} \label{eq:HCA} \Kfun(x, y) = \Kfun(x, \Ytilde) \Kfun(\Xtilde, \Ytilde)^{-1} \Kfun(\Xtilde, y) \end{equation}
where the nodes $\Xtilde$ and $\Ytilde$ are interpolation nodes of an interpolation of $\Kfun(x, y)$ built iteratively.  
Similarly, in their second version of the Hybrid cross approximation algorithm, B{\" o}rm and Grasedyck \cite{borm2005hybrid} propose applying ACA to the kernel matrix evaluated at interpolation nodes to obtain pivots $\Xtilde_{i}$, $\Ytilde_{j}$, and implicitly build an approximation of the form given in \autoref{eq:HCA}. Both those algorithms resemble our approach in that they compute pivots $\Xtilde, \Ytilde$ in some way and then use \autoref{eq:HCA} to build the low-rank approximation. In contrast, our algorithm uses weights, and has stronger accuracy guarantees. We highlight those differences in \autoref{sec:exp}.

Our method inserts itself amongst those low-rank kernel factorization techniques. However, with the notable exception of ACA, those methods often either rely on analytic expressions for the kernel function (and are then limited to some specific ones), or have suboptimal complexities, i.e., greater than $\OO{n r}$. In addition, even though we use interpolation nodes, it is worth noting that our method differs from interpolation-based algorithm as we never explicitly build the $S(x, \Xtilde)$ and $T(y, \Ytilde)$ matrices containing the basis functions. We merely rely on their existence.

$\mathcal{H}$-matrices \cite{hackbusch1999sparse, hackbusch2002data, hackbusch2000sparse} are one way to deal with kernel matrices arising from boundary integral equations that are Hierarchically Block Low-Rank. The compression criterion (i.e., which blocks are compressed as low-rank and which are not) leads to different methods, usually denoted as strongly-admissible (only compress well-separated boxes) or weakly-admissible (compress adjacent boxes as well). In the realm of strongly-admissible $\mathcal{H}$-matrices, the technique of Ho \& Ying \cite{ho2015hierarchical} as well as Tyrtyshnikov \cite{tyrtyshnikov2000incomplete} are of particular interest for us. They use Skeletonization of the matrix to reduce storage and computation cost. In \cite{ho2015hierarchical}, they combine Skeletonization and Sparsification to keep compressing blocks of $\mathcal{H}$-matrices. \cite{tyrtyshnikov2000incomplete} uses a somewhat non-traditional Skeletonization technique to also compress hierarchical kernel matrices.

Finally, extending the framework of low-rank compression, \cite{corona2017tensor} uses tensor-train compression to re-write $\Kfun(X, Y)$ as a tensor with one dimension per coordinate, i.e., $\Kfun(x_1, \dots, x_d, y_1, \dots, y_d)$ and then compress it using the tensor-train model.

\subsection{Contribution}

\subsubsection{Overview of the method} \label{sec:overview}

In this paper, we present a new algorithm that performs this low-rank factorization at a cost proportional to $\OO{m + n}$. The main advantages of the method are as follows:
\begin{enumerate}
  \item The complexity of our method is $\OO{r(m+n)}$ (in terms of kernel function $\Kfun$ evaluations) where $r$ is the target rank.
  \item The method is robust and accurate, irrespective of the distribution of points $x$ and $y$.
  \item We can prove both convergence and numerical stability of the resulting algorithm.
  \item The method is very simple and relies on well-optimized BLAS3 (GEMM) and LAPACK (RRQR, LU) kernels.
\end{enumerate}

Consider the problem of approximating $\Kfun(x,y)$ over the mesh $X \times Y$ with $X \in \Xspace$ and $Y \in \Yspace$. Given the matrix $\Kfun(X, Y)$, one possibility to build a low-rank factorization is to do a rank-revealing LU. This would lead to the selection of
\[ X_{\text{piv}} \subset X, \qquad Y_{\text{piv}} \subset Y \]
called the ``pivots'', and the low-rank factorization would then be given by
\[ \Kfun(X, Y) \approx \Kfun(X, Y_{\text{piv}}) \Kfun(X_{\text{piv}}, Y_{\text{piv}})^{-1} \Kfun(X_{\text{piv}}, Y) \]
In practice however, this method may become inefficient as it requires assembling the matrix $\Kfun(X, Y)$ first.

In this paper, we propose and analyze a new method to select the ``pivots'' \emph{outside} of the sets $X$ and $Y$. The key advantage is that this selection is independent from the sets $X$ and $Y$, hence the reduced complexity. Let us consider the case where $\Xspace, \Yspace = [-1, 1]^d$. We will keep this assumption throughout this paper. Then, within $[-1,1]^d$, one can build tensor grids of Chebyshev points $\Xcheb, \Ycheb$ and associated integration weights $\Wx, \Wy$ and then consider the matrix
\[ K_w = \diag(\Wx)^{1/2} \Kfun(\Xcheb, \Ycheb) \diag(\Wy)^{1/2} \]
Denote $r_0 = \min(|\Xcheb|, |\Ycheb|)$. Based on interpolation properties, we will show that this matrix is closely related to the continuous kernel $\Kfun(x, y)$. In particular, they share a similar spectrum. Then, we select the sets $\Xcheba \subset \Xcheb$, $\Ycheba \subset \Ycheb$ by performing strong rank-revealing QRs \cite{gu1996efficient} over, respectively, $K_w^\top$ and $K_w$ (this is also called a CUR decomposition):
\[ K_w P_y = Q_y R_y \]
\[ K_w^\top P_x = Q_x R_x \]
and build $\Xcheba$ by selecting the elements of $P_x$ associated to the largest rows of $R_x$ and similarly for $\Ycheba$ (if they differ in size, extend the smallest). We denote the rank of this factorization $r_1 = |\Xcheba| = |\Ycheba|$, and in practice, we observe that $r_1 \approx r_{SVD}$, where $r_{SVD}$ is the rank the truncated SVD of $\Kfun(X,Y)$ would provide.
The resulting factorization is
\begin{equation} \label{eq:facto_1} \Kfun(X,Y) \approx \Kfun(X, \Ycheba) \Kfun(\Xcheba, \Ycheba)^{-1} \Kfun(\Xcheba, Y) \end{equation}
Note that, in this process, at no point did we built any Lagrange basis function associated with $\Xcheb$ and $\Ycheb$. We only evaluate the kernel $\Kfun$ at $\Xcheb \times \Ycheb$.

This method appears to be very efficient in selecting sets $\Xcheba$ and $\Ycheba$ of minimum sizes. Indeed, instead, one could aim for a simple interpolation of $\Kfun(x,y)$ over both $\Xspace$ and $\Yspace$ separately. For instance, using the regular polynomial interpolation at Chebyshev nodes $\Xcheb$ and $\Ycheb$, it would lead to a factorization of the form
\[ \Kfun(X,Y) \approx S(X,\Xcheb) \Kfun(\Xcheb,\Ycheb) T(Y,\Ycheb)^\top \]
In this expression, we collect the Lagrange basis functions (each one associated to a node of $\Xcheb$) evaluated at $X$ in the columns of $S(X,\Xcheb)$ and similarly for $T(Y,\Ycheb)$. This provides a robust way of building a low-rank approximation. The rank $r_0 = \min(|\Xcheb|,|\Ycheb|)$, however, is usually much larger than the true rank $r_{SVD}$ and than $r_1$ (given a tolerance). Note that even if those factorizations can always be further recompressed to a rank $\approx r_{SVD}$, they incur a high upfront cost because of the rank $r_0 \gg r_{SVD}$. See \autoref{sec:opt} for a discussion about this.

\subsubsection{Distinguishing features of the method}

Since there are many methods that resemble our approach, we point out its distinguishing features. The singular value decomposition (SVD) offers the optimal low-rank representation in the 2-norm. However, its complexity scales like $\OO{n^3}$. In addition, we will show that the new approach is negligibly less accurate than the SVD in most cases.

The rank-revealing QR and LU factorization, and methods of random projections~\cite{halko2011finding}, have a reduced computational cost of $\OO{n^2 r}$, but still scale quadratically with $n$.

Methods like ACA~\cite{bebendorf2003adaptive}, the rank-revealing LU factorization with rook pivoting~\cite{golub2012matrix} and techniques that randomly sample from columns and rows of the matrix scale like $\OO{n r}$, but they provide no accuracy guarantees. In fact, counterexamples can be found where these methods fail. In contrast, our approach relies on Chebyshev nodes, which offers strong stability and accuracy guarantees. The fact that new interpolation points, $\Xcheb$ and $\Ycheb$, are introduced (the Chebyshev nodes) in addition to the existing points in $X$ and $Y$ is one of the key elements.

Analytical methods are available, like the fast multipole method, etc., but they are limited to specific kernels. Other techniques, which are more general, like Taylor expansion and Chebyshev interpolation~\cite{fong2009black}, have strong accuracy guarantees and are as general as the method presented. However, their cost is much greater; in fact, the difference in efficiency is measured directly by the reduction from $r_0$ to $r_1$ in our approach.

\subsubsection{Low-rank approximation based on SVD and interpolation}

Consider the kernel function $\Kfun$ and its singular value decomposition \cite[theorem VI.17]{reed1980methods}:
\begin{theorem}[Singular Value Decomposition]
Suppose $\Kfun : [-1,1]^d \times [-1,1]^d$ is square integrable. Then there exist two sequences of orthogonal functions $\{u_i\}_{i=1}^\infty$ and $\{v_i\}_{i=1}^d$ and a non-increasing sequence of non-negative real number $\{s_i\}_{i=1}^\infty$ such that
\begin{equation} \label{eqsvd0}
  \Kfun(x, y) = \sum_{s=1}^\infty \sigma_s u_s(x) v_s(y)
\end{equation}
\end{theorem}
As one can see, under relatively mild assumptions, any kernel function can be expanded into a singular value decomposition. Hence from any kernel function expansion we find a low-rank decomposition for the matrix $\Kfun(X, Y)$ (which is \emph{not} the same as the matrix SVD):
\begin{equation} \label{eqsvd}
\Kfun(X, Y) \approx \sum_{s=1}^r u_s(X) \sigma_s v_s(Y) =
\begin{bmatrix} u_1(X) & \cdots & u_r(X) \end{bmatrix}
\begin{bmatrix} \sigma_1 & & \\ & \ddots & \\ & & \sigma_r \end{bmatrix}
\begin{bmatrix} v_1^\top(Y) \\ \vdots \\ v_r^\top(Y) \end{bmatrix}
\end{equation}
where the sequence $\{s_i\}_{i=1}^\infty$ was truncated at an appropriate index $r$. As a general rule of thumb, the smoother the function $\Kfun(x, y)$, the faster the decay of the $\sigma_s$'s and the lower the rank.

If we use a polynomial interpolation method with Chebyshev nodes, we get a similar form:
\begin{equation} \label{eqcheb}
  \Kfun(X, Y) \approx S(X, \Xcheb) \; \Kfun(\Xcheb, \Ycheb) \; T(Y, \Ycheb)^\top
\end{equation}

The interpolation functions $S(x, \Xcheb)$ and $T(y, \Ycheb)$ have strong accuracy guarantees, but the number of terms required in the expansion is $r_0 \gg r \approx r_1$. This is because Chebyshev polynomials are designed for a broad class of functions. In contrast, the SVD uses basis functions $u_s$ and $v_s$ that are optimal for the chosen $\Kfun$.

\subsubsection{Optimal interpolation methods}

\label{sec:opt}

We will now discuss a more general problem, then derive our algorithm as a special case. Let's start with understanding the optimality of the Chebyshev interpolation. With Chebyshev interpolation, $S(x,\Xcheb)$ and $T(y,\Ycheb)$ are polynomials. This is often considered one of the best (most stable and accurate) ways to interpolate smooth functions. We know that for general polynomial interpolants we have:
\begin{equation} \label{eq:err}
f(x) - S(x,\Xcheb) f(\Xcheb) = \frac{f^{(m)}(\xi)}{m!}
\prod_{j=1}^m (x-\Xcheb_j)  
\end{equation}
If we assume that the derivative $f^{(m)}(\xi)$ is bounded, we can focus on finding interpolation points such that
\[ \revrev{\prod_{j=1}^m (x-\Xcheb_j) = x^m - r_{\Xcheb}(x)} \]
is minimal, where \revrev{$r_{\Xcheb}(x)$} is a degree $m-1$ polynomial. Since we are free to vary the interpolation points \revrev{$\Xcheb$}, then we have $m$ parameters (the location of the interpolation points) and $m$ coefficients in \revrev{$r_{\Xcheb}$}. By varying the location of the interpolation points, we can recover any polynomial \revrev{$r_{\Xcheb}$}. Chebyshev points are known to solve this problem optimally. That is, they lead to an $r_{\Xcheb}$ such that \revrev{$\max_x |x^m - r_{\Xcheb}(x)|$} is minimal. 

Chebyshev polynomials are a very powerful tool because of their generality and simplicity of use. Despite this, we will see that this can be improved upon with relatively minimal effort. Let's consider the construction of interpolation formulas for a family of functions $\Kfun(x, \lambda)$, where $\lambda$ is a parameter. We would like to use the SVD, but, because of its high computational cost, we rely on the cheaper rank-revealing QR factorization (RRQR, a QR algorithm with column pivoting). RRQR solves the following optimization problem:
\[ 
\min_{ \{\lambda_s, v_s\} } \max_{\lambda} 
\Big\| \Kfun(x, \lambda) - \sum_{s=1}^m \Kfun(x, \lambda_s) v_s(\lambda) 
\Big\|_2, \qquad 
v_s(\lambda_t) = \delta_{st}
\]
where the 2-norm is computed over $x$---in addition RRQR produces an orthogonal basis for $\{\Kfun(x, \lambda_s)\}_s$ but this is not needed in the current discussion. The vector space span$\{\Kfun(x, \lambda_s)\}_{s=1,\ldots,m}$ is close to span$\{u_s\}_{s=1,\ldots,m}$ [see \autoref{eqsvd0}], and the error  can be bounded by $\sigma_{m+1}$. 

Define $\widehat \Lambda = \{\lambda_1, \dots, \lambda_m\}$.
From there, we identify a set of $m$ interpolation nodes $\Xcheba$ such that the square matrix
\[ 
\Kfun(\Xcheba, \widehat \Lambda) :=
\begin{bmatrix}
\Kfun(\Xcheba, \lambda_1) & \cdots &
\Kfun(\Xcheba, \lambda_m) 
\end{bmatrix}
 \]
is as well conditioned as possible. We now define our interpolation operator as
\[ \Shat(x, \Xcheba) = \Kfun(x, \widehat \Lambda) \Kfun(\Xcheba, \widehat \Lambda)^{-1} \]
By design, this operator is exact on $\Kfun(x, \lambda_s)$:
\[ \Shat(x, \Xcheba) \Kfun(\Xcheba, \lambda_s) = 
\Kfun(x, \lambda_s) \]
It is also very accurate for $\Kfun(x,\lambda)$ since
\[ \Shat(x, \Xcheba) \Kfun(\Xcheba,\lambda)  
\approx \sum_{s=1}^m \Shat(x, \Xcheba) \Kfun(\Xcheba, \lambda_s) v_s(\lambda)  
= \sum_{s=1}^m \Kfun(x, \lambda_s) v_s(\lambda)
\approx \Kfun(x,\lambda)
\]
With Chebyshev interpolation, $S(x, \Xcheb)$ is instead defined using order $m-1$ polynomial functions.

A special case that illustrates the difference between SI and Chebyshev, is with rank-1 kernels:
\[ \Kfun(x, \lambda) = u(x) v(\lambda) \]
In this case, we can pick any $x_1$ and $\lambda_1$ such that $\Kfun(x_1, \lambda_1) \neq 0$, and define $\Xcheba = \{x_1\}$ and
\[ \Shat(x, \Xcheba) = \Kfun(x, \lambda_1) \Kfun(x_1, \lambda_1)^{-1} \]
\[ \Shat(x, \Xcheba) \Kfun(\Xcheba, \lambda) = u(x) v(\lambda_1) \frac 1{u(x_1) v(\lambda_1)} u(x_1) v(\lambda) = u(x) v(\lambda) \]
SI is exact using a single interpolation point $x_1$. An interpolation using Chebyshev polynomials would lead to errors, for any expansion order (unless $u$ is fortuitously a polynomial).

So, one of the key differences between SI and Chebyshev interpolation is that SI uses, as basis for its interpolation, \emph{a set of nearly optimal functions that approximate the left singular functions of $\Kfun$,} rather than generic polynomial functions.

\subsubsection{Proposed method}

In this paper, we use the framework from \autoref{sec:opt} to build an interpolation operator for the class of functions $\Kfun(x,y)$, which we view as a family of functions of $x$ parameterized by $y$ (and vice versa to obtain a symmetric interpolation method). The approximation (\autoref{eq:facto_1}) can be rewritten
\[ \Kfun(X, Y) \approx
\big[ \Kfun(X, \Ycheba) \Kfun(\Xcheba, \Ycheba)^{-1} \big]
\; \Kfun(\Xcheba, \Ycheba) \;
\big[ \Kfun(\Xcheba, \Ycheba)^{-1} \Kfun(\Xcheba, Y) \big]
\]
and by comparing with \autoref{eqcheb}, we recognize the interpolation operators:
\[
  \Shat(x,\Xcheba) = \Kfun(x, \Ycheba) \Kfun(\Xcheba, \Ycheba)^{-1},
  \qquad
  \That(y,\Ycheba) = \Kfun(\Xcheba, y)^\top \Kfun(\Xcheba, \Ycheba)^{-T}
\]
These interpolation operators are nearly optimal; because of the way these operators are constructed we call the method ``Skeletonized Interpolation.'' The sets $\Xcheba$ and $\Ycheba$ are the minimal sets such that if we sample $\Kfun$ at these points we can interpolate $\Kfun$ at any other point with accuracy $\epsilon$. In particular, $\Xcheba$ and $\Ycheba$ are much smaller than their Chebyshev-interpolant counterparts $\Xcheb$ and $\Ycheb$ and their size, $r_1$, is very close to $r$ in \autoref{eqsvd}. The approach we are proposing produces nearly-optimal interpolation functions for our kernel, instead of generic polynomial functions.

Note that none of the previous discussions explains why the proposed scheme is stable; the inverse $\Kfun(\Xcheba, \widehat \Lambda)^{-1}$ as well as $\Kfun(\Xcheba, \Ycheba)^{-1}$ in \autoref{eq:facto_1} could become troublesome numerically. We will explain in detail in \autoref{sec:stab} why this is not an issue numerically, and we explore the connection with interpolation in more detail in \autoref{sec:new_interp}.

\subsubsection{Organization of the paper}

    This paper is organized as follows. In \autoref{sec:algo}, we present the algorithm in detail and present some theoretical results about its convergence. In \autoref{sec:stab}, we discuss its numerical stability and in \autoref{sec:new_interp} we revisit the interpolation interpretation on a simple example. Finally, \autoref{sec:exp} illustrates the algorithm on more complex geometries, compares its accuracy with other classical algorithms and presents computational complexity results.

\section{Skeletonized Interpolation} \label{sec:algo}

\subsection{The algorithm}

\autoref{algo:ai} provides the high-level version of the algorithm. It consists of 3 steps:
\begin{itemize}
    \item Build grids $\Xcheb$ and $\Ycheb$, tensor grids of Chebyshev nodes.
        Over $[-1,1]$ in 1D, the $\overline m$ Chebyshev nodes of the first kind are defined as
        \[ \bar x_k = \cos \left(\frac{2k-1}{2\overline{m}} \pi\right) \qquad k = 1,\dots,\overline{m} \] 
        In higher dimensions, they are defined as the tensor product of one-dimen\-sional grids.
        The number of points in every dimension should be such that
        \[ \sum_{k=1}^\nx \sum_{l=1}^\ny S_k(x) \Kfun(\bar x_k, \bar y_l) T_l(y) = S(x,\Xcheb) \Kfun(\Xcheb, \Ycheb) T(y, \Ycheb)^\top \]
        provides an $\delta$ uniform approximation over $[-1,1]^d \times [-1,1]^d$ of $\Kfun(x, y)$.
        Denote
        \[ r_0 = \min(|\Xcheb|, |\Ycheb|) \]
    \item Recompress the grid by performing a strong rank-revealing QR factorization \cite{gu1996efficient} of
    \begin{equation} \label{eq:Kww} \diag(\Wx)^{1/2} \Kfun(\Xcheb, \Ycheb) \diag(\Wy)^{1/2} \end{equation}
    and its transpose, up to accuracy $\epsilon$. This factorization is also named CUR decomposition \cite{mahoney2009cur, cheng2005compression}. While our error estimates only hold for strong rank-revealing QR \revrev{factorizations}, in practice, a simple column-pivoted QR factorization based on choosing columns with large norms works as well.
    In the case of Chebyshev nodes of the first kind in 1D over $[-1,1]$ the integration weights are given by
        \[ w_k = \frac \pi \nx \sqrt{1 - \bar x_k^2} = \frac \pi \nx \sin \left( \frac{2k-1}{2\nx}\pi \right) \] The weights in $d$ dimensions are the products of the corresponding weights in 1D, and the $\diag(\Wx)$ and $\diag(\Wy)$ matrices are simply the diagonal matrices of the integration weights. Denote
        \[ r_1 = |\Xcheba| = |\Ycheba| \]
        In case the sets $\Xcheba$ and $\Ycheba$ output by those RRQR's are of slightly different size (which we rarely noticed in our experiments), extend the smallest to have the same size as the largest.
    \item  Given $\Xcheba$ and $\Ycheba$, the low-rank approximation is given by
        \[ \Kfun(X, \Ycheba) \Kfun(\Xcheba, \Ycheba)^{-1} \Kfun(\Xcheba, Y) \]
        of rank $r_1 \approx r_{SVD}$.
\end{itemize}

\begin{algorithm}
    \begin{algorithmic}
        \Procedure{Skeletonized Interpolation}{$\Kfun : [-1,1]^d \times [-1,1]^d \to \R$, $X$, $Y$, $\epsilon$, $\delta$}
            \State Build $\Xcheb$ and $\Ycheb$, sets of Chebyshev nodes over $[-1, 1]^d$ that interpolate $\Kfun$ with error $\delta$ uniformly
            \State Build $K_w$ as
            \[ K_w = \diag(\Wx)^{1/2} \Kfun(\Xcheb, \Ycheb) \diag(\Wy)^{1/2} \]
            \State Extract $\Ycheba \subseteq \Ycheb$ by performing a strong RRQR over $K_w$ with tolerance $\epsilon$ ; 
            \[ K_w P_y = Q_y R_y \]
            \State Extract $\Xcheba \subseteq \Xcheb$ by performing a strong RRQR over $K_w^\top$ with tolerance $\epsilon$ ; 
            \[ K_w^\top P_x = Q_x R_x \]
            \State If the sets have different size, extends the smallest to the size of the largest.

            \Return
                \[ \Kfun(X, Y) \approx \Kfun(X, \Ycheba) \Kfun(\Xcheba, \Ycheba)^{-1} \Kfun(\Xcheba, Y) \]
        \EndProcedure
    \end{algorithmic}
    \caption{Skeletonized Interpolation}
    \label{algo:ai}
\end{algorithm}

\subsection{Theoretical Convergence}

\subsubsection{Overview}

In this section, we prove that the error made during the RRQR is not too much amplified when evaluating the interpolant. We first recall that
\begin{enumerate}
    \item From interpolation properties, 
    \[ \Kfun(x,y) = S(x,\Xcheb) \Kfun(\Xcheb,\Ycheb) T(y,\Ycheb)^\top + \Eint(x,y) \]
    where $T$ and $S$ are small matrices (i.e., bounded by logarithmic factors in $r_0$) and $\Eint = \OO{\delta}$.
    \item From the strong RRQR properties,  
    \[ K_w = \begin{bmatrix} I \\ \Shat \end{bmatrix} \Kwhat \begin{bmatrix} I & \That^\top \end{bmatrix} + \Eqr \]
    where $\Kwhat$ has a spectrum similar to \revrev{that} of $K_w$ (up to a small polynomial), $\Shat$ and $\That$ are bounded by a small polynomial, and $\Eqr = \OO{\epsilon}$.
\end{enumerate}
Then, by combining those two facts and assuming $\delta < \epsilon$, one can show
\begin{enumerate}
    \item First, that the interpolation operators are bounded,
    \begin{equation} \label{eq:bound_overview} \| \Kfun(x, \Ycheba) \Kfun(\Xcheba, \Ycheba)^{-1} \|_2 = \OO{p(r_0,r_1)} \end{equation}
        where $p$ is a small polynomial.
    \item Second, that the error $\epsilon$ made in the RRQR is not too much amplified, i.e., 
    \begin{equation} \label{eq:error_overview} | \Kfun(x, y) - \Kfun(x, \Ycheba) \Kfun(\Xcheba, \Ycheba)^{-1} \Kfun(\Xcheba, y) | = \OO{p'(r_0, r_1)\epsilon} \end{equation}
        where $p'$ is another small polynomial.
\end{enumerate}

Finally, if one assume that $\sigma_i(K_w)$ decays exponentially fast, so does $\epsilon$ and the resulting approximation in \autoref{eq:error_overview} converges.

In the following, we present the main lemmas (some proofs are relocated in the appendix for brevity) leading to the above result.

\subsubsection{Interpolation-related results}

We first consider the interpolation itself. Consider $\Xcheb$ and $\Ycheb$, constructed such as
    \[ \Kfun(x, y) = S(x, \Xcheb) \Kfun(\Xcheb, \Ycheb) T(y, \Ycheb)^{\top} + \Eint(x, y) \]
    \begin{lemma}[Interpolation at Chebyshev Nodes]\label{lemma:1}
        $\forall x \in \Xspace$ and $\Xcheb$ tensor grids of Chebyshev nodes of the first kind,
        \begin{gather*}
            \| S(x, \Xcheb) \|_2 = \OO{\log(|\Xcheb|)^d}
        \end{gather*}
        where $\Xspace \subset \R^d$. In addition, the weights, collected in the weight matrix $\DWx$ are such that
    \begin{gather*}
            \| \DWx^{1/2} \|_2 \leq \frac{\pi^{d/2}}{\sqrt{\nx}} = \OO{\frac 1{\sqrt{\nx}}} \\
            \| \DWx^{-1/2} \|_2 \leq \frac{\nx}{\pi^{d/2}} = \OO{\nx}
        \end{gather*}
        where $\nx = |\Xcheb|$.
    \end{lemma}

\subsubsection{Skeletonization results}

We now consider the skeletonization step of the algorithm performed through the two successive rank-revealing QR factorizations.

\paragraph{Rank-Revealing QR factorizations} 

    Let us first recall what a rank-revealing QR factorization is. Given a matrix $A \in \R^{m \times n}$, one can compute a rank-revealing QR factorization \cite{golub2012matrix} of the form
\[ A \Pi = \begin{bmatrix} Q_1 & Q_2 \end{bmatrix} \begin{bmatrix} R_{11} & R_{12} \\ & R_{22} \end{bmatrix} \]
where $\Pi$ is a permutation matrix, $Q$ an orthogonal matrix and $R$ a triangular matrix. Both $R$ and $Q$ are partitioned so that $Q_1 \in \R^{m \times k}$ and $R_{11} \in \R^{k \times k}$. If $\|R_{22}\| \approx \varepsilon$, this factorization typically indicates that $A$ has an $\varepsilon$-rank of $k$.
The converse, however, is not necessarily true \cite{golub2012matrix} in general.

From there, one can also write
\[ A \Pi = Q_1 R_{11} \begin{bmatrix} I & R_{11}^{-1} R_{12} \end{bmatrix} + E = A_1 \begin{bmatrix} I & T \end{bmatrix} + E \]
where $T$ is the \emph{interpolation} operator, $A_1$ a set of $k$ columns of $A$ and $E$ the approximation error. This approximation can be achieved by a simple column-pivoted QR algorithm \cite{golub2012matrix}. This algorithm, however, is not guaranteed to always work (i.e., even if $A$ has rapidly decaying singular values, this rank-revealing factorization may fail to exhibit it). 

A \emph{strong} rank-revealing QR, however, has more properties.
It has been proven \cite{gu1996efficient,cheng2005compression} that one can compute in $\OO{m n^2}$  a rank-revealing QR factorization that guarantees
\begin{equation} \label{eq:strongRRQR} \sigma_i(A_1) \geq \frac{\sigma_i(A)}{q_1(n,k)} \text{, } \sigma_j(E) \leq \sigma_{k+j}(A) q_1(n,k) \text{ and } \|T\|_F \leq q_2(n,k) \end{equation}
    where $q_1$ and $q_2$ are two small polynomials (with fixed constants and degrees). The existence of this factorization is a crucial part of our argument.
Using the interlacing property of singular values \cite{golub2012matrix}, this implies that we now have both lower and upper bounds on the singular values of $A_1$
\begin{equation} \label{eq:lower-upper} \frac{\sigma_i(A)}{q_1(n,k)} \leq \sigma_i(A_1) \leq \sigma_i(A) \end{equation}
From \autoref{eq:strongRRQR} we can directly relate the error $E$ and $\sigma_{k+1}$ from
\begin{equation} \label{eq:error-lower} \|E\|_2 = \sigma_1(E) \leq \sigma_{k+1}(A) q_1(n,k) \end{equation}

Finally, given a matrix $A$, one can apply the above result to both its rows and columns, leading to a factorization
\[ \Pi_r^\top A \Pi_c = \begin{bmatrix} I \\ T_r \end{bmatrix} A_{rc} \begin{bmatrix} I & T_c \end{bmatrix} + E \]
with the same properties as detailed above.

\paragraph{Skeletonized Interpolation}
We can now apply this results to the $K_w$ and $\Kwhat$ matrices. 

    \begin{lemma}[CUR Decomposition of $K_w$] \label{lemma:3}
        The partition $\Xcheb = \Xcheba \cup \Xchebb$, $\Ycheb = \Ycheba \cup \Ychebb$ of \autoref{algo:ai} is such that there exist $\widecheck S$, $\widecheck T$, $\Eqr(\Xcheb, \Ycheb)$ matrices and a slowly-growing polynomial $p(r_0, r_1)$ such that
            \[ K_w = \begin{bmatrix} I \\ \widecheck S \end{bmatrix} \Kwhat \begin{bmatrix} I & \widecheck{T}^\top \end{bmatrix} + \Eqr(\Xcheb, \Ycheb) \]
        and where
        \begin{gather*}
        \epsilon = \| \Eqr(\Xcheb, \Ycheb) \|_2 \leq  p(r_0, r_1) \sigma_{r_1+1}(K_w) \\
            \| \widecheck S \|_2 \leq p(r_0, r_1) \\
            \| \widecheck T \|_2 \leq p(r_0, r_1)
        \end{gather*}
        Finally, we have
        \[ \| \Kwhat^{-1} \|_2  \leq \frac{p(r_0, r_1)^2}{\epsilon} \]
    \end{lemma}

    \begin{proof}
        The first three results are direct applications of \cite[theorem 3 and remark 5]{cheng2005compression} as explained in the previous paragraph.
        The last result follows from the properties of the strong rank-revealing QR:
        \[ \|\hat K_w^{-1}\|_2 = \frac 1{\sigma_{r_1}(\hat K_w)} \leq \frac{p(r_0,r_1)}{\sigma_{r_1}(K_w)} \leq \frac{p(r_0,r_1)}{\sigma_{r_1+1}(K_w)} \leq \frac{p(r_0,r_1)^2}{\epsilon} \]
        The first inequality follows from $\sigma_{r_1}(K_w) \leq \sigma_{r_1}(\Kwhat) p(r_0,r_1)$ (\autoref{eq:lower-upper}), the second from $\sigma_{r_1}(K_w) \geq \sigma_{r_1+1}(K_w)$ (by definition of singular values) and the last from $\sigma_{r_1+1}(K_w)^{-1} \leq p(r_0,r_1) \epsilon^{-1}$ (\autoref{eq:error-lower}).

    \end{proof}

    Finally, a less obvious result
    \begin{lemma} \label{lemma:4}
        There exist a polynomial $q(r_0, r_1)$ such that for any $x \in \Xspace, y \in \Yspace$,
        \begin{gather*}
            \| \Kfun(x, \Ycheba) \Kfun(\Xcheba, \Ycheba)^{-1} \|_2 = \OO{q(r_0, r_1)} \\
            \| \Kfun(\Xcheba, \Ycheba)^{-1} \Kfun(\Xcheba, y) \|_2 = \OO{q(r_0, r_1)}
        \end{gather*}
    \end{lemma}
    We provide the proof in the appendix; the key ingredient is simply that $\|\hat K_w^{-1}\|_2 \leq p(r_0,r_1)^2 \epsilon^{-1}$ from the RRQR properties; hence $\hat K_w$ is ill-conditioned, but not arbitrarily. Its condition number grows like $\epsilon^{-1}$. Then, when multiplied by quantities like $\epsilon$ or $\delta \ll \epsilon$, the factors cancel out and the resulting product can be bounded.

\subsubsection{Link between the node matrix and the continuous SVD} \label{sec:contsvd}

In this section, we link the continuous SVD and the spectrum (singular values) of the matrix $\DWx^{1/2} K_w \DWy^{1/2}$. \revrev{This justifies the use of the weights.}

For the sake of simplicity, consider the case where interpolation is performed at \emph{Gauss-Legendre} nodes $\Xcheb, \Ycheb$ with the corresponding integration weights $\Wx, \Wy$.
(A more complete explanation can be found in \cite{yarvin1998generalized}.)

Take the classical discrete SVD of $K_w$,
    \[ K_w = \overline U\, \overline \Sigma\, \overline V^\top \]
We then have
\[ \Kfun(x, y) = \underbrace{ S_w(x, \Xcheb) \overline U \, \overline \Sigma \, \overline V^\top T_w(y, \Ycheb)^\top }_{=\overline \Kfun(x,y)}+ \Eint(x, y) \]
Then, denote the sets of new basis functions
    \[ \overline u(x) = S_w(x, \Xcheb) \overline U \qquad \overline v(y) = T_w(y, \Ycheb) \overline V \]
The key is to note that those functions are orthonormal. Namely, for $\overline u$,
\begin{align*} 
\int_\Xspace \overline u_i(x) \overline u_j(x) \dif x & = \sum_{k=1}^{r_0} \overline w_k \overline u_i(\overline x_k) \overline u_j(\overline x_k) \\
& = \sum_{k=1}^{r_0} \overline w_k \left( \sum_{l=1}^{r_0} \overline w_l^{-1/2} S_l(\overline x_k) \overline U_{li} \right) \left( \sum_{l=1}^{r_0} \overline w_l^{-1/2} S_l(\overline x_k) \overline U_{lj} \right)  \\
& = \sum_{k=1}^{r_0} \overline w_k \left(  \sum_{l=1}^{r_0} \delta_{kl} \overline w_l^{-1/2} \overline U_{li} \right) \left( \sum_{l=1}^{r_0} \delta_{kl} \overline w_l^{-1/2} \overline U_{lj} \right) \\
& = \sum_{k=1}^{r_0} \overline w_k \overline w_k^{-1/2} \overline U_{ki} \overline w_k^{-1/2} \overline U_{kj} = \sum_{k=1}^{r_0} \overline U_{ki} \overline U_{kj} = \delta_{ij}
    \end{align*}
The same result holds for $\overline v$. 
This follows from the fact that a Gauss-Legendre quadrature rule with $n$ points can exactly integrate polynomials up to degree $2n-1$.
This shows that we are implicitly building a factorization
\begin{equation} \label{eq:implicit} \Kfun(x, y) = \sum_{s=1}^\infty \sigma_s u_s(x) v_s(y) = \underbrace{\sum_{s=1}^{r_0} \sigma_s(K_w) \overline u_s(x) \overline v_s(y)}_{=\overline \Kfun(x,y)} + \Eint(x, y) \end{equation}
where the approximation error is bounded by the interpolation error $\Eint$ and where the sets of basis functions are orthogonal.

    Assume now that the kernel $\Kfun$ is square-integrable over $[-1,1]^d \times [-1,1]^d$. This is called a Hilbert-Schmidt kernel \cite[Lemma 8.20]{renardy2006introduction}. This implies that the associated linear operator is compact \cite[Theorem 8.83]{renardy2006introduction}. $\KfunI$ is compact as well since it is finite rank \cite[Theorem 8.80]{renardy2006introduction}. Given the fact that $|\Eint(x,y)| \leq \delta$ for all $x,y$, $\|\Kfun - \KfunI\|_{L_2} \leq C \delta$ for some $C$ and hence, by compactness of both operators \cite[Corollary 2.2.14]{guven2016quantitative},
    \[ |\sigma_i - \sigma_i(\KfunI)|  \leq C \delta \]
    for some $C > 0$. Then, from the above discussion, we clearly have $\sigma_i(K_w) = \sigma_i(\KfunI) + \OO{\delta}$ and hence
    \[ \sigma_i(K_w) = \sigma_i + \OO{\delta} \]

This result only formally holds for Gauss-Legendre nodes and weights. However, this motivates the use of integration weights in the case of Chebyshev as well.

\subsubsection{Convergence of the Skeletonized interpolation} \label{sec:theorem}

We \revrev{now} present the main result of this paper:

\begin{theorem}[Convergence of Skeletonized Interpolation] \label{thm:ai} If $\Xcheba$ and $\Ycheba$ are constructed following \autoref{algo:ai}, then there exist a polynomial $r(r_0, r_1)$ such that for any $x \in \Xspace$ and $y \in \Yspace$,
    \[ | \Kfun(x,y) - \Kfun(x,\Ycheba) \Kfun(\Xcheba, \Ycheba)^{-1} \Kfun(\Xcheba, y) | = \OO{\epsilon \, r(r_0, r_1)} \]
\end{theorem}

The key here is that the error incurred during the CUR decomposition, $\epsilon$ is amplified by, at most, a polynomial of $r_0$ and $r_1$. Hence, \autoref{thm:ai} indicates that if the spectrum decays fast enough (i.e., if $\epsilon \to 0$ when $r_0, r_1 \to \infty$ faster than $r(r_0,r_1)$ grows), the proposed approximation should converge to the true value of $\Kfun(x, y)$.

What is simply left is then linking $\epsilon$, $r_0$, and $r_1$. We have, from the CUR properties,
\[\epsilon \leq p(r_0,r_1) \sigma_{r_1+1}(K_w) \]
which implies
    \[ |\Kfun(x,y) - \Kfun(x, \Ycheba) \Kfun(\Xcheba,\Ycheba)^{-1} \Kfun(\Xcheba,y)| = \OO{\sigma_{r_1+1}(K_w) r'(r_0,r_1)} \]

Then, following the discussion from \autoref{sec:contsvd}, we expect 
\[ \sigma_i(K_w) = \sigma_i + \OO{\delta} \]
Hence, if $\Kfun$ has rapidly-decaying singular values, so does $K_w$. Assuming the singular values of $K_w$ decay exponentially fast, i.e.,
    \[ \log \sigma_k(K_w) \approx \text{poly}(k), \]
we find
    \[ |\Kfun(x, y) - \Kfun(x, \Ycheba) \Kfun(\Xcheba, \Ycheba)^{-1} \Kfun(\Xcheba, y)| \to 0 \]
as $r_0,r_1 \to \infty$, or alternatively, as $\epsilon \to 0$.

\section{Numerical stability} \label{sec:stab}

\subsection{The problem}

The previous section indicates that, at least theoretically, we can expect convergence as $\epsilon \to 0$. However, the factorization
\begin{equation} \label{eq:SI} \Kfun(X,Y) \approx \Kfun(X, \Ycheba) \Kfun(\Xcheba, \Ycheba)^{-1} \Kfun(\Xcheba, Y) \end{equation}
    seems to be numerically challenging to compute. Indeed, as we showed in the previous section, we can only really expect at best $\|\Kwhat^{-1}\|_2 = \OO{\epsilon^{-1}}$ which indicates that, roughly,
\[ \kappa(\Kfun(\Xcheba,\Ycheba)) = \OO{\frac 1 \epsilon} \]
i.e., the condition number grows with the desired accuracy, and convergence beyond a certain threshold (like $10^{-8}$ in double-precision) seems impossible. Hence, we can reasonably be worried about the numerical accuracy of computing \autoref{eq:SI} even with a stable algorithm.

Note that this is not a pessimistic upper bound; by construction, $\hat K_w$ really is ill-conditioned, and experiments show that solving linear systems $\hat K_w x = b$ with random right-hand sides is numerically challenging and leads to errors of the order $\epsilon^{-1}$.

\subsection{Error Analysis}
Consider \autoref{eq:SI} and let for simplicity
\[ K_x = \Kfun(X, \Ycheba), \quad 
   K = \Kfun(\Xcheba, \Ycheba), \quad 
   K_y = \Kfun(\Xcheba, Y) \]
In this section, our goal is to show why one can expect this formula to be accurately computed if one uses backward stable algorithms. As proved in \autoref{sec:algo}, we have the following bounds on the interpolation operators
\begin{gather*}
  \| K_x K^{-1} \|_2 \leq p(r_0, r_1) \\
  \| K^{-1} K_y \|_2 \leq p(r_0, r_1)
\end{gather*}
for some polynomial $p$. The key is that there is no $\epsilon^{-1}$ in this expression.
Those bounds essentially follow from the guarantees provided by the strong rank-revealing QR algorithm.

\revrev{
    Now, let's compute the derivative of $K_x K^{-1} K_y$ with respect to $K_x$, $K$ and $K_y$ \cite{petersen2008matrix}:
    \begin{align*} \partial(K_x K^{-1} K_y) & = (\partial K_x) K^{-1} K_y + K_x (\partial (K^{-1})) K_y  + K_x K^{-1} (\partial K_y) \\
                                            & = (\partial K_x) K^{-1} K_y - K_x K^{-1} (\partial K) K^{-1} K_y  + K_x K^{-1} (\partial K_y) \end{align*}
    }

\revrev{
Then, consider perturbing $K_x$, $K$, $K_y$ by $\varepsilon$ (assume all matrices are of order $\OO{1}$ for the sake of simplicity), i.e., let $\delta K_x$, $\delta K_y$ and $\delta K$ be perturbations of $K_x$, $K_y$ and $K$, respectively, with
\[ \|\delta K_x\| = \OO{\varepsilon}, \|\delta K_y\| = \OO{\varepsilon}, \|\delta K\| = \OO{\varepsilon}. \]
Then, using the above derivative as a first order approximation, we can write
\begin{align*} \| & K_x K^{-1} K - (K_x + \delta K_x)(K + \delta K)^{-1}(K_y + \delta K_y) \| \\
                    & \leq \|\delta K_x\| \|K^{-1} K_y\| + \|K_x K^{-1}\| \|\delta K\| \|K^{-1} K_y\| + \|K_x K^{-1}\| \|\delta K_y\| + \OO{\varepsilon^2} \\ 
                    & \leq 2 \epsilon p(r_0, r_1) + \epsilon p(r_0, r_1)^2 + \OO{\varepsilon^2} \\
                    & = \varepsilon (2 p(r_0,r_1) + p(r_0,r_1)^2) + \OO{\varepsilon^2} \\
\end{align*}
We see that the computed result is independent of the condition number of $K = \Kfun(\Xcheba, \Ycheba)$ and depends on $p(r_0, r_1)$ only.}

Assume now that we are using backward stable algorithms in our calculations \cite{higham2002accuracy}. We then know that the computed result is the result of an exact computation where the inputs have been perturbed by $\varepsilon$. The above result indicates that the numerical result (with roundoff errors) can be expected to be accurate up to $\varepsilon$ times a small polynomial, hence stable.

\section{Skeletonized Interpolation as a new interpolation rule}

\label{sec:new_interp}

As indicated in the introduction, one can rewrite
\begin{align*} \Kfun(x,y) & \approx \Kfun(x, \Ycheba) \Kfun(\Xcheba, \Ycheba)^{-1} \Kfun(\Xcheba, y) \\
                          & = \left[ \Kfun(x, \Ycheba) \Kfun(\Xcheba, \Ycheba)^{-1} \right] \Kfun(\Xcheba, \Ycheba) \left[ \Kfun(\Xcheba, \Ycheba)^{-1} \Kfun(\Xcheba, y) \right]  \\
                          & = \Shat(x, \Xcheba) \Kfun(\Xcheba, \Ycheba) \That(y, \Ycheba)^\top \end{align*}
where we recognize two new ``cross-interpolation" (because they are build by considering both the $\Xspace$ and $\Yspace$ space) operators $\Shat(x, \Xcheba) = \Kfun(x, \Ycheba) \Kfun(\Xcheba, \Ycheba)^{-1}$ and $\That(y, \Ycheba) = \Kfun(\Xcheba, y)^\top \Kfun(\Xcheba, \Ycheba)^{-\top}$. In this notation, each column of $\Shat(x, \Xcheba)$ and $\That(y, \Ycheba)$ is a Lagrange function associated to the corresponding node in $\Xcheba$ or $\Ycheba$ and evaluated at $x$ or $y$, respectively.

This interpretation is interesting as it allows to ``decouple'' $x$ and $y$ and analyze them independently. In particular, one can look at the quality of the interpolation of the basis functions $u_k(x)$ and $v_k(y)$ using $\Shat$ and $\That$. Indeed, if this is accurate, it is easy to see that the final factorization is accurate. Indeed,
\begin{align*} \Kfun(x, y) & \approx \sum_{k=1}^r \sigma_k u_k(x) v_k(y) \\
                           & \approx  \sum_{k=1}^r \sigma_k (\Shat(x, \Xcheba) u_k(\Xcheba)) (\That(y, \Ycheba) v_k(\Ycheba))^\top \\
                           & =       \Shat(x, \Xcheba) \left( \sum_{k=1}^r \sigma_k u_k(\Xcheba) v_k(\Ycheba)^\top \right) \That(y, \Ycheba)^\top  \\
                           & \approx \Shat(x, \Xcheba) \Kfun(\Xcheba, \Ycheba) \That(y, \Ycheba)^\top \\
                           & \approx \Kfun(x, \Ycheba) \Kfun(\Xcheba, \Ycheba)^{-1} \Kfun(\Xcheba, y) \\
                           \end{align*}

To illustrate this, let us consider a simple 1-dimensional example. Let $x, y \in [-1,1]$ and consider
\[ \Kfun(x, y) = \frac 1{4 + x - y}. \]
Then approximate this function up to $\epsilon = 10^{-10}$, and obtain a factorization of rank $r$.

\autoref{fig:toy_lag4} illustrates the 4\textsuperscript{th} Lagrange basis function in $x$, i.e., $\Shat(x, \Xcheba)_4$ and the classical Lagrange polynomial basis function associated with the same set $\Xcheba$. We see that they are both 1 at $\Xcheba_4$ and 0 at the other \revrev{points}. However, $\Shat(x, \Xcheba)_4$ is much more stable and small than its polynomial counterpart. In the case of polynomial interpolation at equispaced nodes, the growth of the Lagrange basis function (or, equivalently, of the Lebesgue constant) is the reason for the inaccuracy and instability.

\autoref{fig:toy_interp_umax} shows the effect of interpolating $u_r(x)$ using $\Shat(x, \Xcheba)$ as well as using the usual polynomial interpolation at the nodes $\Xcheba$. We see that $\Shat(x, \Xcheba)$ interpolates very well $u_r(x)$, showing indeed that we implicitly build an accurate interpolant, even on the last (least smooth) eigenfunctions. The usual polynomial interpolation fails to capture any feature of $u_r$ on the other hand. 
Note that we could have reached a similar accuracy using interpolation at Chebyshev nodes but only by using many more interpolation nodes.

Finally, \autoref{fig:toy_err} shows how well we approximate the various $r$ eigenfunctions. As one can see, interpolation is very accurate on $u_1(x)$, but the error grows for less smooth eigenfunctions. The growth is, roughly, similar to the growth of $\frac{\epsilon}{\sigma_i}$. Notice how this is just enough so that the resulting factorization is accurate:
\begin{align*}
    \Shat(x, \Xcheba) \Kfun(\Xcheba, y) &= \sum_{s=1}^r \sigma_s \Shat(x, \Xcheba) u_s(\Xcheba) v_s(y) +\OO{\epsilon} \\
                                        &= \sum_{s=1}^r \sigma_s \left(u_s(x) + \OO{\frac{\epsilon}{\sigma_s}}\right) v_s(y) +\OO{\epsilon} \\
                                        &= \sum_{s=1}^r \sigma_s u_s(x) v_s(y) + \sum_{s=1}^r \OO{\epsilon} v_s(y) + \OO{\epsilon} \\
                                    &= \Kfun(x, y) + \OO{\epsilon} \end{align*}
It is also consistent with the analysis of \autoref{sec:stab}. This illustrates how the algorithm works: it builds an interpolation scheme that allows for interpolating the various eigenfunctions of $\Kfun$ with just enough accuracy so that the resulting interpolation is accurate up to the desired accuracy. 

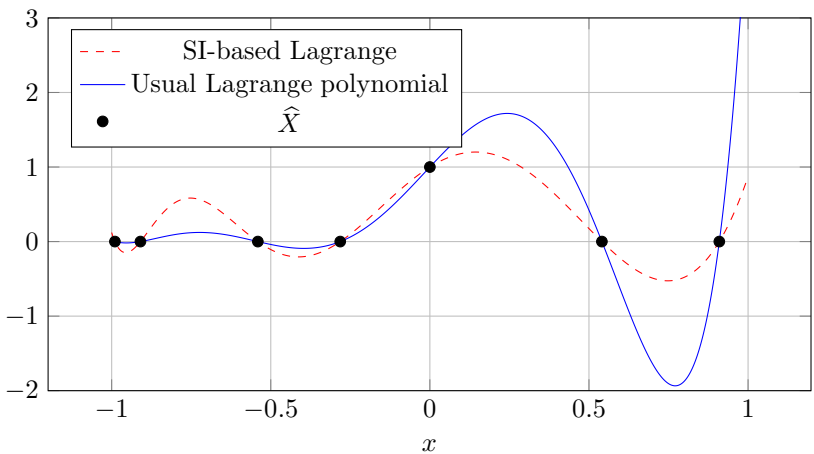
\begin{figure}[htbp]
    \centering
    \begin{tikzpicture}
        \begin{axis}
            [
              xlabel={$x$},
              height=0.5\textwidth,
              width=0.9\textwidth,
              grid=major,
              legend pos=north west,
              ymax=3,
              ymin=-2,
              xtick={-1, -0.5, 0, 0.5, 1},
              legend entries={SI-based Lagrange, Usual Lagrange polynomial, $\Xcheba$}
            ]
            \addplot [no marks, dashed, red] table [x=x,y=Shat]{figs/Lag4_X.dat};
            \addplot [no marks, blue] table [x=x,y=Spoly]{figs/Lag4_X.dat};
            \addplot [only marks, black] table [x=Xhat,y=OneShot]{figs/Lag4_Xhat.dat};
        \end{axis}
    \end{tikzpicture}
    \caption{4\textsuperscript{th} Lagrange basis function. We see that the Chebyshev-SI based Lagrange basis function is more stable than the usual polynomial going through the same interpolation nodes.}
    \label{fig:toy_lag4}
\end{figure}

\begin{figure}[htbp]
    \centering
    \begin{tikzpicture}
        \begin{axis}
            [
              xlabel={$x$},
              height=0.5\textwidth,
              width=0.8\textwidth,
              grid=major,
              legend style={cells={align=left}},
              legend pos=outer north east,
              ymax=2.5,
              ymin=-2.5,
              xtick={-1, -0.5, 0, 0.5, 1},
              legend style={row sep=3pt},
              legend entries={SI-based\\interpolant, Polynomial\\interpolant, Singular\\function, $\Xcheba$}
            ]
            \addplot [no marks, dashed, red] table [x=x,y=Shatu]{figs/Uend_X.dat};
            \addplot [no marks, orange] table [x=x,y=ShatLowu]{figs/Uend_X.dat};
            \addplot [no marks, dashed, blue] table [x=x,y=u]{figs/Uend_X.dat};
            \addplot [only marks, black] table [x=Xhat,y=u]{figs/Uend_Xhat.dat};
        \end{axis}
    \end{tikzpicture}
    \caption{Interpolation of the last (and least smooth) eigenfunction. We see that the Chebyshev-SI based interpolant is much more accurate than the polynomial interpolant going through the same interpolation nodes.}
    \label{fig:toy_interp_umax}
\end{figure}
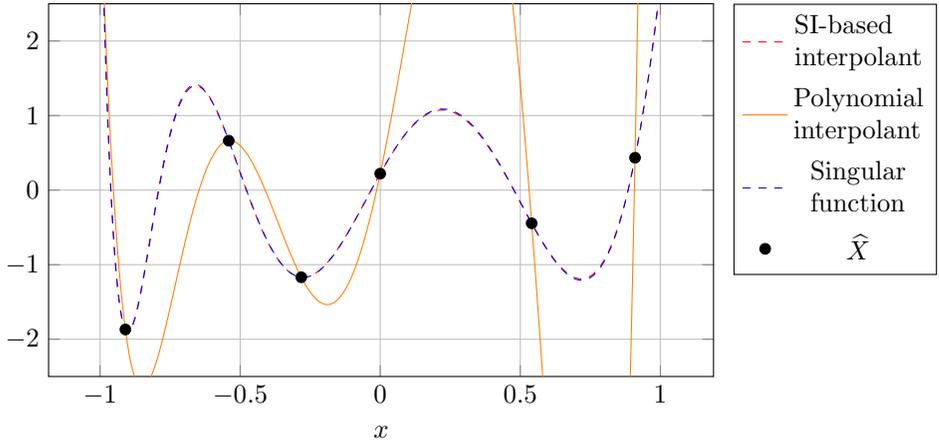

\begin{figure}[htbp]
    \centering
    \begin{tikzpicture}
        \begin{semilogyaxis}
            [
              xlabel={Eigenfunction},
              ylabel={Error},
              height=0.5\textwidth,
              width=0.7\textwidth,
              grid=major,
              legend entries={Interpolation error on $u_i$, $\epsilon/\sigma_i$},
              legend pos=north west,
            ]
            \addplot [red, mark=diamond] table [x=i,y=errShatui]{figs/errShatui.dat};
            \addplot [blue, mark=*] table [x=i,y=epsSigmai]{figs/errShatui.dat};
        \end{semilogyaxis}
    \end{tikzpicture}
    \caption{Interpolation error on the various eigenfunctions. The error grows just slowly enough with the eigenfunctions so that the overall interpolant is accurate up to the desired accuracy.}
    \label{fig:toy_err}
\end{figure}
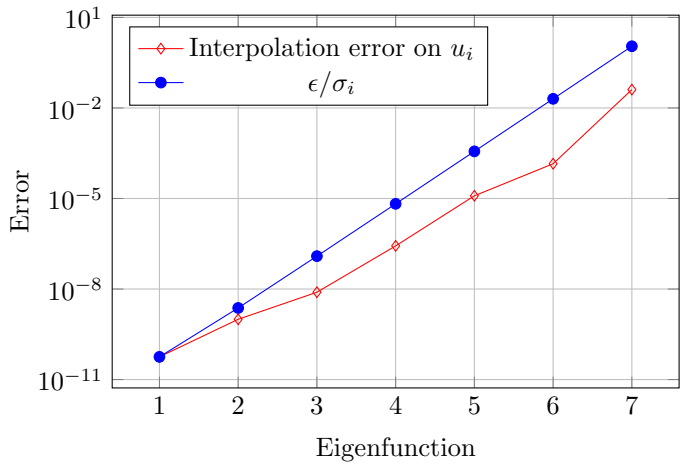

\section{Numerical experiments} \label{sec:exp}

In this section we present some numerical experiments on various geometries.
We study the quality (how far $r_1$ is from the optimal SVD-rank $r$ and how accurate the approximation is) of the algorithm in \autoref{sec:basic} and \autoref{sec:comp}. We illustrate in \autoref{sec:aca_rrqr} the improved guarantees of RRQR and justify the use of weights in \autoref{sec:weights}. Finally, \autoref{sec:time} studies the algorithm computational complexity.

The experiments are done using Julia \cite{bezanson2017julia} and the code is sequential. For the strong rank-revealing QR algorithm, we use the \texttt{LowRankApprox.jl} Julia package \cite{kenneth_l_ho_2018_1254148}. The code can be downloaded from \url{https://people.stanford.edu/lcambier/publications}.

\subsection{Simple geometries} \label{sec:basic}

We begin this section with an elementary problem, as depicted on \autoref{sfig:basic_1}. In this problem, we consider the usual kernel $\Kfun(x, y) = \|x - y\|_2^{-1}$ where $x, y \in \R^2$. $\Xspace$ and $\Yspace$ are two squares of side of length 1, centered at $(0.5,0.5)$ and $(2.5,2.5)$ respectively. Finally, $X$ and $Y$ are two uniform meshes of $50 \times 50$ mesh points each, i.e., $n = 2500$.

\revrev{
We pick the Chebyshev grids $\Xcheb$ and $\Ycheb$ using a heuristic based on the target accuracy $\varepsilon$. Namely, we pick the number of Chebyshev nodes in each dimension (i.e., $x_1$, $x_2$, $y_1$ and $y_2$) independently (by using the midpoint of $\Xspace$ and $\Yspace$ as reference), such that the interpolation error is approximately less than $\varepsilon^{3/4}$. This values is heuristic, but performs well for those geometries. Other techniques are possible. 
This choice is based in part on the observation that the algorithm is accurate even when $\delta > \varepsilon$, i.e., when the Chebyshev interpolation is \emph{less accurate} than the actual final low-rank approximation through Skeletonized Interpolation.
}

Consider then \autoref{sfig:basic_2}.
The $r_0$ line indicates the rank ($r_0 = \min(|\Xcheb|, |\Ycheb|)$) of the low-rank expansion through interpolation. The $r_1$ line corresponds to the rank obtained after the RRQR over $\Kfun(\Xcheb, \Ycheb)$ and its transpose, i.e., it is the rank of the final approximation
\[ \Kfun(X, Y) \approx \Kfun(X, \Ycheba) \Kfun(\Xcheba, \Ycheba)^{-1} \Kfun(\Xcheba, Y) \]
Finally, ``SVD rank'' is the rank one would obtain by truncating the SVD of \[\Kfun(X,Y)=USV^\top\] at the appropriate singular value, as to ensure \[ \|\Kfun(X, Y) - \tilde \Kfun(X, Y)\|_F \approx \epsilon \|\Kfun(X, Y)\|_F\] 
Similarly, ``RRQR'' is the rank a rank-revealing QR on $\Kfun(X, Y)$ would obtain. This is usually slightly suboptimal compared to the SVD. Those two values are there as to illustrate that $r_1$ is close to the optimal value.

The conclusion regarding \autoref{sfig:basic_2} is that Skeletonized Interpolation is nearly optimal in terms of rank. While the rank obtained by the interpolation is clearly far from optimal, the RRQR over $\Kfun(\Xcheb, \Ycheb)$ allows us to find subsets $\Xcheba \subset \Xcheb$ and $\Ycheba \subset \Ycheb$ that are enough to represent well $\Kfun(X, Y)$, and the final rank $r_1$ is nearly optimal compared to the SVD-rank $r$.  We also see that the rank \revrev{of} a blind RRQR over $\Kfun(X, Y)$ is higher than the SVD-rank and usually closer---if not identical---to $r_1$.

We want to re-emphasize that, in practice, \emph{the error of the sets $\Xcheb$, $\Ycheb$ ---i.e., the error of the polynomial interpolation based on $\Xcheb \times \Ycheb$--- can be larger than the required tolerance}. If they are large enough, the compressed sets $\Xcheba$, $\Ycheba$ will contain enough information so as to properly interpolate $\Kfun$ and the final error will be smaller than the required tolerance. This is important, as the size of the Chebyshev grid for a given tolerance can be fairly large (as indicated in the introduction, and one of the main motivation of this work), even though the final rank is small.

As a sanity check, \autoref{sfig:basic_3} gives the relative error measured in the Frobenius norm
\[ \frac{\|\Kfun(X, Y) - \Kfun(X, \Ycheba) \Kfun(\Xcheba, \Ycheba)^{-1} \Kfun(\Xcheba, Y)\|_F}{\|\Kfun(X, Y)\|_F} \]
between $\Kfun(X, Y)$ and its interpolation as a function of the tolerance $\epsilon$.\footnote{Choosing the Frobenius norm is not critical---very similar results are obtained in the 2-norm for instance.} We see that both lines are almost next to each other, meaning our approximation indeed reaches the required tolerance. This is important as it means that one can effectively \emph{control} the accuracy.

Finally, \autoref{sfig:basic_1} also shows the resulting $\Xcheba$ and $\Ycheba$. It is interesting to notice how the selected points cluster near the close corners, as one could expect since this is the area where the kernel is the least smooth.

\begin{figure}[htbp]
    \centering
    \subfloat[\label{sfig:basic_2}Ranks as a function of the desired accuracy]{
        \begin{tikzpicture}
        \begin{semilogxaxis}
        [
          xlabel={Tolerance},
          ylabel={Rank},
          height=0.4\textwidth,
          width=0.9\textwidth,
          grid=major,
          legend entries={$r_0$, $r_1$, RRQR, SVD},
          legend pos=north east,
          xtick={1e-12, 1e-9, 1e-6, 1e-3},
          xticklabels={$10^{-12}$, $10^{-9}$, $10^{-6}$, $10^{-3}$},
        ]
        \addplot table [x=tol,y=r0]{figs/exp1.dat};
        \addplot table [x=tol,y=r1]{figs/exp1.dat};
        \addplot table [x=tol,y=rRRQR]{figs/exp1.dat};
        \addplot table [x=tol,y=rSVD]{figs/exp1.dat};
        \end{semilogxaxis}
        \end{tikzpicture}
    } \\
    \subfloat[\label{sfig:basic_1}The geometry used, and the resulting choice of $\Xcheba$, $\Ycheba$ for a tolerance of $10^{-6}$.]{
        \begin{tikzpicture}
        \begin{axis}
        [
          xlabel={$x_1$},
          ylabel={$x_2$},
          height=0.4\textwidth,
          width=0.4\textwidth,
          legend entries={$X$, $Y$},
          legend pos=south east,
        ]
        \addplot [red, mark=square*, only marks] table [x=x1,y=x2]{figs/exp1-xy.dat};
        \addplot [blue, mark=diamond*, only marks] table [x=y1,y=y2]{figs/exp1-xy.dat};
        \draw (axis cs:0,0)--(axis cs:0,1)--(axis cs:1,1)--(axis cs:1,0)--(axis cs:0,0);
        \draw (axis cs:2,2)--(axis cs:2,3)--(axis cs:3,3)--(axis cs:3,2)--(axis cs:2,2);
        \end{axis}
        \end{tikzpicture}
    } \;\;
    \subfloat[\label{sfig:basic_3}Relative Frobenius-norm error as a function of the desired accuracy]{
        \begin{tikzpicture}
        \begin{loglogaxis}
        [
          xlabel={Tolerance},
          ylabel={Relative Frobenius error},
          height=0.4\textwidth,
          width=0.4\textwidth,
          grid=major,
          legend entries={Error, Tolerance},
          legend style={at={(1.3,0.4)}}
        ]
        \addplot table [x=tol,y=tol]{figs/exp1.dat};
        \addplot table [x=tol,y=err]{figs/exp1.dat};
        \end{loglogaxis}
        \end{tikzpicture}
    }
    \caption{Results for the 2D-squares example. The rank $r_0$ before compression is significantly reduced to $r_1$, very close to the true SVD or RRQR-rank.}
\end{figure}
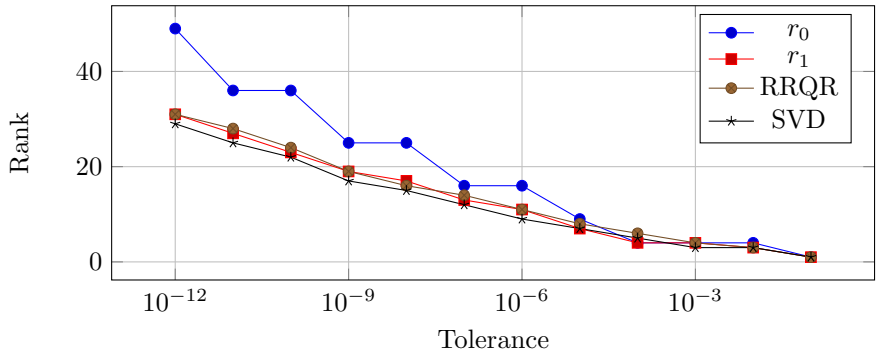
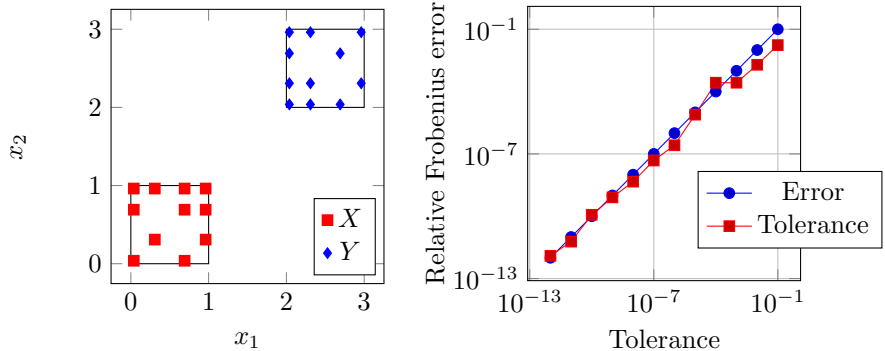

We then consider results for the same Laplacian kernel $\Kfun(x, y) = \|x-y\|_2^{-1}$ between two plates in 3D (\autoref{sfig:basic_4}).
We observe overall very similar results as for the previous case on \autoref{sfig:basic_5}, where the initial rank $r_0$ is significantly decreased to $r_1$ while keeping the resulting accuracy close to the required tolerance as \autoref{sfig:basic_6} shows.
Finally, one can see on \autoref{sfig:basic_4} the selected Chebyshev nodes. They again cluster in the areas where smoothness is the worst, i.e., at the closes edges of the plates.

\begin{figure}[htbp]
    \centering
    \subfloat[\label{sfig:basic_5}Ranks as a function of the desired accuracy]{
        \begin{tikzpicture}
        \begin{semilogxaxis}
        [
          xlabel={Tolerance},
          ylabel={Rank},
          height=0.4\textwidth,
          width=0.9\textwidth,
          legend entries={$r_0$, $r_1$, RRQR, SVD},
          legend pos=north east,
          grid=major,
          xtick={1e-12, 1e-9, 1e-6, 1e-3},
          xticklabels={$10^{-12}$, $10^{-9}$, $10^{-6}$, $10^{-3}$},
        ]
        \addplot table [x=tol,y=r0]{figs/exp2.dat};
        \addplot table [x=tol,y=r1]{figs/exp2.dat};
        \addplot table [x=tol,y=rRRQR]{figs/exp2.dat};
        \addplot table [x=tol,y=rSVD]{figs/exp2.dat};
        \end{semilogxaxis}
        \end{tikzpicture}
    } \\
    \subfloat[\label{sfig:basic_4}The geometry used, and the resulting choice of $\Xcheba$, $\Ycheba$ for a tolerance of $10^{-6}$.]{
        \begin{tikzpicture}
        \begin{axis}
        [
          xlabel={$x_1$},
          ylabel={$x_2$},
          zlabel={$x_3$},
          height=0.4\textwidth,
          width=0.4\textwidth,
          legend entries={$X$, $Y$},
          legend style={at={(1.3,0.4)}}
        ]
        \addplot3 [red, mark=square*, only marks] table [x=x1,y=x2,z=x3]{figs/exp2-xy.dat};
        \addplot3 [blue, mark=diamond*, only marks] table [x=y1,y=y2,z=y3]{figs/exp2-xy.dat};
        \draw (axis cs:-1,0,0)--(axis cs:-1,1,0)--(axis cs:-1,1,1.5)--(axis cs:-1,0,1.5)--(axis cs:-1,0,0);
        \draw (axis cs:0,0,-1)--(axis cs:0,1,-1)--(axis cs:1.5,1,-1)--(axis cs:1.5,0,-1)--(axis cs:0,0,-1);
        \end{axis}
        \end{tikzpicture}
    } \;\;
    \subfloat[\label{sfig:basic_6}Relative Frobenius-norm error as a function of the desired accuracy]{
        \begin{tikzpicture}
        \begin{loglogaxis}
        [
          xlabel={Tolerance},
          ylabel={Relative Frobenius error},
          height=0.4\textwidth,
          width=0.4\textwidth,
          grid=major,
          legend entries={Error, Tolerance},
                    legend style={at={(1.2,0.35)}}
        ]
        \addplot table [x=tol,y=tol]{figs/exp2.dat};
        \addplot table [x=tol,y=err]{figs/exp2.dat};
        \end{loglogaxis}
        \end{tikzpicture}
    }
    \caption{Results for the perpendicular plates example. The rank $r_0$ before compression is significantly reduced to $r_1$, very close to the true SVD or RRQR-rank.}
\end{figure}
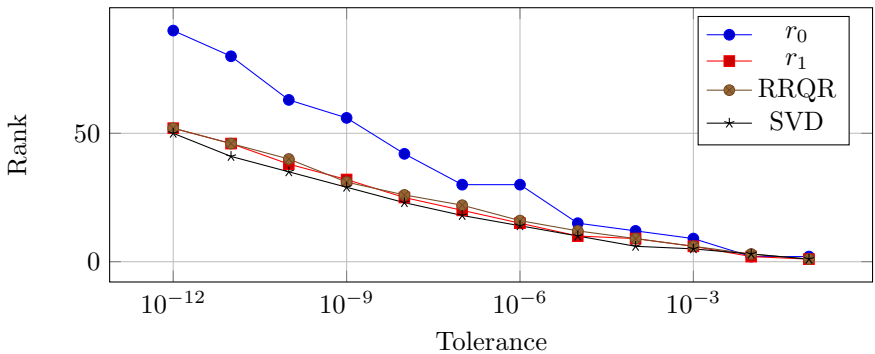
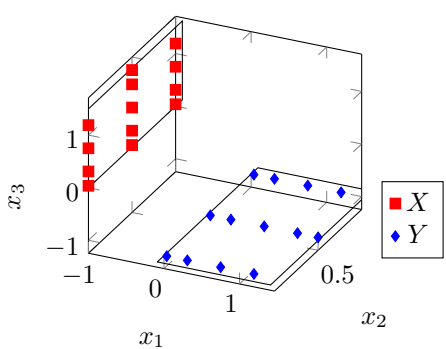
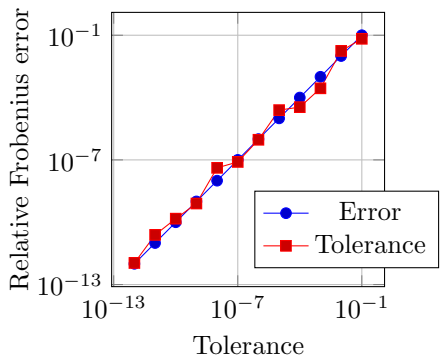

\subsection{Comparison with ACA and Random Sampling} \label{sec:comp}

We then compare our method with other standard algorithms for kernel matrix factorization. In particular, we compare it with ACA \cite{bebendorf2003adaptive} and 'Random CUR' where one selects, at random, pivots $\Xtilde$ and $\Ytilde$ and builds a factorization \[ \Kfun(X,Y) \approx \Kfun(X, \Ytilde) \Kfun(\Xtilde, \Ytilde)^{-1} \Kfun(\Xtilde, Y) \] based on those. As we are interested in comparing the \emph{quality} of the resulting sets of pivots for a given rank, we compare those algorithms for sets $X$ and $Y$ with variable distance between each other and for a fixed tolerance ($\epsilon = 10^{-8}$) and kernel ($\Kfun(x,y) = \|x-y\|_2^{-1}$). The geometry is two unit-length squares side-by-side with a variable distance between their closest edges.

The comparison is done in the following way:
\begin{enumerate}
    \item Given $r_1$, build the ACA factorization of rank $r_1$ and compute its relative error in Frobenius norm with $\Kfun(X, Y)$;
    \item Given $r_1$, build the random CUR factorization by sampling uniformly at random $r_1$ points from $X$ and $Y$ to build $\Xtilde$, $\Ytilde$. Then, compute its relative error with $\Kfun(X, Y)$.
\end{enumerate}
We then do so for sets of varying distance, and for a given distance, we repeat the experiment 25 times by building $X$ and $Y$ at random within the two squares. This allows to study the variance of the error and to collect some statistics.

\autoref{fig:comp} gives the resulting errors in relative Frobenius norm for the 3 algorithm using box-plots of the errors to show distributions. 
The rectangular boxes represent the distributions from the 25\% to the 75\% quantiles, with the median in the center. The thinner lines represent the complete distribution, except outliers depicted using large dots.
We observe that the ($\Xcheba$, $\Ycheba$) sets based on Chebyshev-SI are, \emph{for a common size $r_1$}, \emph{more accurate} than the Random or ACA sets. In addition, by design, they lead to more stable factorizations (as they have very small variance in terms of accuracy) while ACA for instance has a higher variance. We also see, as one may expect, that while ACA is still fairly stable even when the clusters get close, random CUR starts having higher and higher variance. This is understandable as the kernel gets less and less smooth as the clusters get close.

Finally, we ran the same experiments with several other kernels (${r^{-2}}$, ${r^{-3}}$, $\log(r)$, $\exp(-r)$, $\exp(-r^2)$) and observed quantitatively very similar results.

\begin{figure}[htbp]
    \centering
    \includegraphics[width=\textwidth]{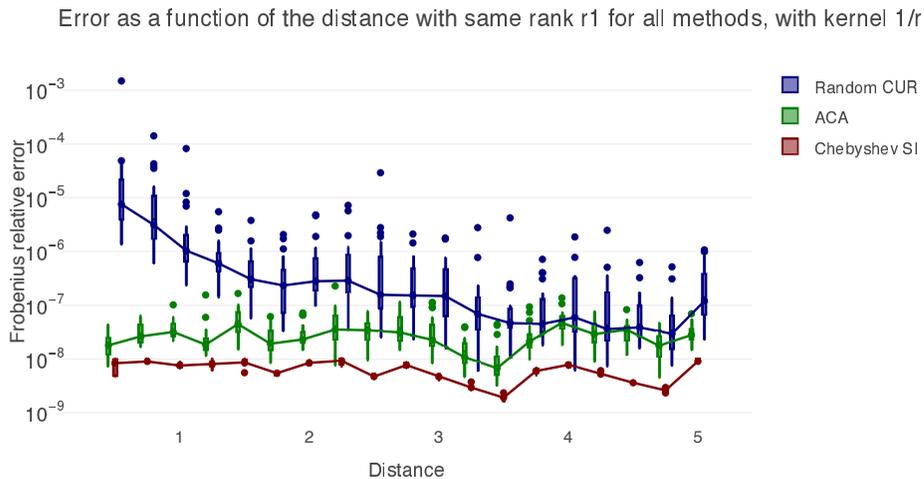}
    \caption{Comparison between different algorithms: Chebyshev-based SI, ACA and purely random CUR decomposition. We consider two 2D squares of sides 1 with a variable distance from each other; for each distance, we run Chebyshev-based SI and find the smallest sets $\Xcheb, \Ycheb$ of rank $r_0$ leading to a factorization using $\Xcheba, \Ycheba$ of sizes $r_1$ with relative error at most $10^{-8}$. Then $r_1$ is used as an a priori rank for ACA and Random CUR. We randomize the experiments by subsampling 500 points from a large $100 \times 100$ points grid in each square.}
    \label{fig:comp}
\end{figure}

\subsection{Stability guarantees provided by RRQR} \label{sec:aca_rrqr}
In \autoref{algo:ai}, in principle, any rank-revealing factorization providing pivots could be used. In particular, ACA itself could be used. In this case, this is the HCAII (without the weights) algorithm as described in \cite{borm2005hybrid}. However, ACA is only a heuristic: unlike strong rank-revealing factorizations, it can't always reveal the rank.  In particular, it may have issues when some parts of $X$ and $Y$ have strong interactions while others are weakly coupled. To highlight this, consider the following example. It can be extended to many other situations.

Let us use the rapidly decaying kernel
\[ \Kfun(x, y) = \frac{1}{\|x-y\|_2^3} \]
and the situation depicted in \autoref{geo:guarantees} with $X = \begin{bmatrix} X_1 & X_2 \end{bmatrix}$ and $Y = \begin{bmatrix} Y_1 & Y_2 \end{bmatrix}$.
\revrev{We note that, formally, $X$ and $Y$ are not well-separated.}

Since $\Kfun$ is rapidly decaying and $X_1$/$Y_2$ (resp.\ $X_2$/$Y_1$) are far away, the resulting matrix is \emph{nearly} block diagonal, i.e.,
\begin{equation} \label{eq:fail} \Kfun(X, Y) \approx \begin{bmatrix} \Kfun(X_1, Y_1) & \OO{\varepsilon} \\ \OO{\varepsilon} & \Kfun(X_2, Y_2) \end{bmatrix} \end{equation}
for some small $\varepsilon$. This is a challenging situation for ACA since it will need to sweep through the initial block completely before considering the other one. In practice heuristics can help alleviate the issue; see ACA+~\cite{borm2005hybrid} for instance. Those heuristic, however, do not come with any guarantees. Strong RRQR, on the other hand, does not suffer from this and picks optimal nodes in each cluster from the start. It guarantees stability and convergence.

\begin{figure}
\centering
    \subfloat[The geometry. $L_y = 0.9 L_x$, each domain $X_i, Y_i$ has 50 points uniformly distributed (hence, the weights are uniform) on an arc of angle $\pi/4$. $\Xcheb = X$ and $\Ycheb = Y$, and $\Xcheba$, $\Ycheba$ are $r_1$ points subsampled from $\Xcheb, \Ycheb$ using \autoref{algo:ai}.]{
\raisebox{0.5cm}{\begin{tikzpicture}
    \draw [line width=2] (2,0) arc (-20:20:2.5cm);
    \draw (2.5,0.8) node{$X_1$};
    \draw [line width=2](1.7,0.1) arc (-20:20:2cm);
    \draw (1.4,0.8) node{$Y_1$};
    \draw [line width=2](-2,0) arc (20:-20:-2.5cm);
    \draw (-2.5,0.8) node{$X_2$};
    \draw [line width=2](-1.7,0.1) arc (20:-20:-2cm);
    \draw (-1.4,0.8) node{$Y_2$};
    \draw[<->] (-1.7,-0.1) -- (1.7,-0.1);
    \draw[<->] (-2,1.9) -- (2,1.9);
    \draw (0,2.1) node{$L_x$};
    \draw (0,-0.4) node{$L_y$};
\end{tikzpicture}}
} \quad
\subfloat[Relative Frobenius error over $X \times Y$ using both RRQR and ACA to select $\Xcheba, \Ycheba$ of size $r_1$ from $\Xcheb$, $\Ycheb$ using \autoref{algo:ai}.] {
\begin{tikzpicture}
\begin{semilogyaxis}
[
  xlabel={$r_1 = |\Xcheba| = |\Ycheba|$},
  ylabel={Error},
  height=0.3\textwidth,
  width=0.5\textwidth,
  legend entries={RRQR, ACA},
  legend pos=south west,
]
\addplot table [x=rank,y=err_rrqr]{figs/rrqr_vs_aca.dat};
\addplot table [x=rank,y=err_aca]{figs/rrqr_vs_aca.dat};
\end{semilogyaxis}
\end{tikzpicture}
}
\caption{Failure of ACA. The geometry is such that the coupling between $X_1$/$Y_1$ and $X_2$/$Y_2$ is much stronger than between $X_1$/$Y_2$ and $X_2$/$Y_1$. This leads to ACA not selecting pivots properly. RRQR on the other hand has no issue and converges steadily.}
\label{geo:guarantees}
\end{figure}
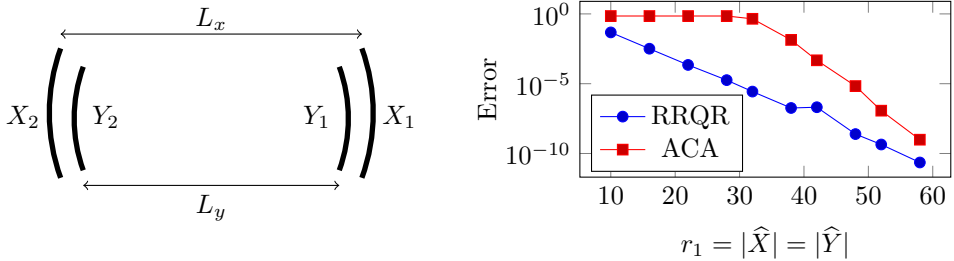

\subsection{The need for weights} \label{sec:weights}

Another characteristic of \autoref{algo:ai} is the presence of weights. We illustrate here why this is necessary in general. \autoref{algo:ai} uses $\Xcheb$ and $\Ycheb$ both to select interpolation points (the ``columns" of the RRQR) and to evaluate the resulting error (the ``rows"). Hence, a non-uniform distribution of points leads to over- or under-estimated $L_2$ error and to a biased interpolation point selection. The weights, roughly equal to the (square-root of) the inverse points density, alleviate this effect. This is a somewhat small effect in the case of Chebyshev nodes \& weights as the weights have limited amplitudes.

To illustrate this phenomenon more dramatically, consider the situation depicted on \autoref{fig:weights_no_weights}. We define $\Xcheb$ and $\Ycheb$ in the following way. Align two segments of $N$ points, separated by a small interval of length $\epsilon = 1/N$ \revrev{with $N = 200$}. At the close extremities we insert $25 N$ additional points inside small 3D spheres of diameter $\epsilon$. As a result, $|\Xcheb| = |\Ycheb| = 26N = 5,200$, and $\Xcheb$, $\Ycheb$ are strongly non-uniform. We see that the small spheres hold a large number of points in an interval of length $N^{-1}$. As a result, their associated weight should be proportional to $N^{-1/2}$, while the weight for the points on the segments should be proportional to $1$. Then we apply \autoref{algo:ai} with and without weights, and evaluate the error on the segments using $|X| = |Y| = 10,000$ equispaced points as a proxy for the $L_2$ error.

When using a rank $r_0 = 200$, the CUR decomposition picks only 6 more points on the segments (outside the spheres) for the weighted case compared to the unweighted. However, this is enough to dramatically improve the accuracy, as \autoref{fig:r1_w_no_w} shows. Overall, the presence of weights has a large effect, and this shows that in general, one should appropriately weigh the node matrix $K_w$ to ensure maximum accuracy.

\begin{figure}
\centering
\subfloat[The geometry. Kernel is $1/r$, $\epsilon =0.01$.]{
    \raisebox{1cm}{\begin{tikzpicture}
    \pgfmathsetmacro{\l}{0.3}
    \pgfmathsetmacro{\d}{0.6}
    \draw (0,0) circle (\l);
    \draw (\l,0) -- (5*\l,0);
    \draw (-4*\l,0) circle (\l);
    \draw (-5*\l,0) -- (-9*\l,0);
    \draw [<->] (-5*\l,-\d) -- (-3*\l,-\d);
    \draw [<->] (-3*\l,-\d) -- (-\l,-\d);
    \draw [<->] (-\l,-\d) -- (\l,-\d);
    \draw (-4*\l,-1.4*\d) node {$\epsilon$};
    \draw (-2*\l,-1.4*\d) node {$\epsilon$};
    \draw (0,-1.4*\d) node {$\epsilon$};
    \draw [<->] (-9*\l,\d) -- (-3*\l,\d);
    \draw [<->] (-1*\l,\d) -- (5*\l,\d);
    \draw (-6*\l,1.4*\d) node {$1$};
    \draw (2*\l,1.4*\d) node {$1$};
    \draw (-7*\l, -0.7*\d) node {$\Xspace$};
    \draw (3*\l, -0.7*\d) node {$\Yspace$};
\end{tikzpicture}}
} \;
\subfloat[Error with and without weights in \autoref{algo:ai}.\label{fig:r1_w_no_w}] {
\begin{tikzpicture}
\begin{semilogyaxis}
[
  xlabel={$r_1 = |\Xcheba| = |\Ycheba|$},
  ylabel={Error},
  height=0.3\textwidth,
  width=0.45\textwidth,
  ymax=1000,
]
\addplot [red,mark=diamond*,mark size=1pt] table [x=rank, y=no_weights] {figs/weights_no_weights.dat};
\addplot [blue,mark=square*,mark size=0.5pt] table [x=rank, y=with_weights] {figs/weights_no_weights.dat};
\node at (5,3) [anchor=west] {\textcolor{red}{No weights}};
\node at (5,-14) [anchor=west] {\textcolor{blue}{With weights}};
\end{semilogyaxis}
\end{tikzpicture}
}
\caption{Benchmark demonstrating the importance of using weights in the RRQR factorization. The setup for the benchmark is described in the text. The blue curve on the right panel, which uses weights, has much improved accuracy.}
\label{fig:weights_no_weights}
\end{figure}
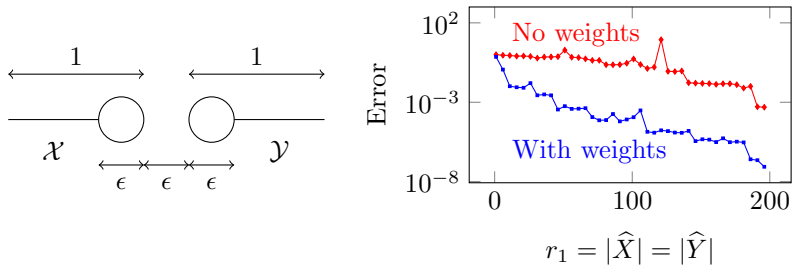

\subsection{Computational complexity} \label{sec:time}

We finally study the computational complexity of the algorithm.
It's important to note that two kinds of operations are involved: kernel evaluations and classical flops. As they may potentially differ in cost, we keep those separated in the following analysis.

The cost of the various parts of the algorithm is the following :
\begin{itemize}
    \item $\OO{r_0^2}$ kernel evaluations for the interpolation, i.e., the construction of $\Xcheb$ and $\Ycheb$ and the construction of $\Kfun(\Xcheb, \Ycheb)$
    \item $\OO{r_0^2 r_1}$ flops for the RRQR over $\Kfun(\Xcheb, \Ycheb)$ and $\Kfun(\Xcheb, \Ycheb)^\top$
    \item $\OO{(m + n)r_1}$ kernel evaluations for computing $\Kfun(X, \Ycheba)$ and $\Kfun(\Xcheba, Y)$, respectively (with $m = |X|$ and $n = |Y|$)
    \item $\OO{r_1^3}$ flops for $\Kfun(\Xcheba, \Ycheba)^{-1}$ (through, say, an LU factorization)
\end{itemize}
So the total complexity of building the three factor is $\OO{(m + n) r_1}$ kernel evaluations. If $m = n$ and $r_1 \approx r$, the total complexity is
\[ \OO{(m + n) r_1} \approx \OO{n r} \]
Also note that the memory requirements are, clearly, of order $\OO{(m+n)r_1}$.

When applying this low-rank matrix on a given input vector $f(Y) \in \R^n$, the cost is
\begin{itemize}
    \item $\OO{r_1 n}$ flops for computing $w_1 = \Kfun(\Xcheba, Y) f(Y)$
    \item $\OO{r_1^2}$ flops for computing $w_2 = \Kfun(\Xcheba, \Ycheba)^{-1} w_1$ assuming a factorization of \\$\Kfun(\Xcheba, \Ycheba)$ has already been computed
    \item $\OO{m r_1}$ flops for computing $w_3 = \Kfun(X, \Ycheba) w_2$
\end{itemize}
So the total cost is
\[ \OO{(m+n)r_1} \approx \OO{n r} \]
flops if $m = n$ and $r_1 \approx r$.

To illustrate those results, \autoref{sfig:complexity_1} shows, using the same setup as in the 2D square example of \autoref{sec:basic}, the time (in seconds) taken by our algorithm versus the time taken by a naive algorithm that would first build $\Kfun(X, Y)$ and then perform a rank-revealing QR on it. Time is given as a function of $n$ for a fixed accuracy $\epsilon = 10^{-8}$.
One should not focus on the absolute values of the timing but rather the asymptotic complexities. In this case, the $\OO{n}$ and $\OO{n^2}$ complexities clearly appear, and our algorithm scales much better than the naive one (or, really, that any algorithm that requires building the full matrix first).
Note that we observe no loss of accuracy as $n$ grows. Also note that the plateau at the beginning of the Skeletonized Interpolation curve is all the overhead involved in selecting the Chebyshev points $\Xcheb$ and $\Ycheb$ using some heuristic. This is very implementation-dependent and could be reduced significantly with a better or more problem-tailored algorithm. However, since this is by design independent of $X$ and $Y$ (and, hence, $n$) it does not affect the asymptotic complexity.

\autoref{sfig:complexity_2} shows the time as a function of the desired accuracy $\epsilon$, for a fixed number of mesh points $n = 10^5$. Since the singular values of $\Kfun(X, Y)$ decay exponentially, one has $r \approx \OO{\log\left(\frac 1\epsilon\right)}$. The complexity of the algorithm being $\OO{nr}$, we expect the time to be proportional to $\log\left(\frac 1\epsilon\right)$. This is indeed what we observe.

\autoref{sfig:complexity_3} depicts the time as a function of the rank $r$ for a fixed accuracy $\epsilon = 10^{-8}$ and number of mesh points $n = 10^5$. In that case, to vary the rank and keep $\epsilon$ fixed, we change the geometry and observe the resulting rank. This is done by moving the top-right square (see \autoref{sfig:basic_1}) towards the bottom-left one (keeping approximately one cluster diameter between them) or away from it (up to 6 diameters). The rank displayed is the rank obtained by the factorization. As expected, the algorithm scales linearly as a function of $r$.

\begin{figure}[htbp]
    \centering
    \subfloat[\label{sfig:complexity_1}Time as a function of $n$ for a fixed tolerance. The plateau is the overhead in the Skeletonized Interpolation algorithm that is independent of $n$.]{
    \begin{tikzpicture}
        \begin{loglogaxis}
        [
          xlabel={$n = |X| = |Y|$},
          ylabel={Time [s.]},
          height=0.45\textwidth,
          width=0.9\textwidth,
          grid=major,
          legend entries={Naive algorithm, SI, $\OO{n^2}$, $\OO{n}$},
          legend pos=north west,
        ]
        \addplot [red, mark=diamond*] coordinates {(10, 2e-4) (30, 4e-4) (90, 2e-3) (250, 1.5e-2) (600, 1e-1)   (1800, 1)};
        \addplot [blue, mark=square*] coordinates {(10, 3e-3) (30, 4e-3) (90, 5e-3) (250, 6e-3)   (600, 1.2e-2) (1800, 3e-2) (4500, 8e-2) (16000, 2e-1) (40000, 8e-1) (100000, 1.1)};
        \addplot [domain=10:1e5, black] {x^2/1e5};
        \addplot [domain=10:1e5, dashed, black] {x/3e+5};
        \addplot [domain=2000:1e5, red, dashed] {x^2/2.9e6};
        \end{loglogaxis}
    \end{tikzpicture} 
    }  \\
    \subfloat[\label{sfig:complexity_2}Time as a function of $\epsilon$, for fixed clusters of points.]{
    \begin{tikzpicture}
        \begin{loglogaxis}
        [
          xlabel={Tolerance},
          ylabel={Time [s.]},
          height=0.45\textwidth,
          width=0.45\textwidth,
          grid=major,
          legend entries={SI, $\log(1/\varepsilon)$}
        ]
        \addplot coordinates {(1e-6, 1.8) (1e-5, 0.9) (1e-4, 0.7) (1e-3,0.47) (1e-2,0.47) (1e-1, 0.1)};
        \addplot [domain=1e-6:1e-1] {-log10(x)/4.3};
        \end{loglogaxis}
    \end{tikzpicture}
    } \,\,
    \subfloat[\label{sfig:complexity_3}Time as a function of $r$. The rank is varied by increasing the distance between the two clusters, for a fixed tolerance and number of points.]{
    \begin{tikzpicture}
        \begin{axis}
        [
          xlabel={Rank},
          ylabel={Time [s.]},
          height=0.45\textwidth,
          width=0.45\textwidth,
          grid=major,
          ytick={0.6, 1, 1.5},
        ]
        \addplot coordinates {(6,0.65) (6,0.61) (7, 0.78) (8, 0.9) (9,1.11) (13,1.55)};
        \end{axis}
    \end{tikzpicture}
    }
    \caption{Timings experiments on Skeletonized Interpolation}
    \label{fig:complexity}
\end{figure}
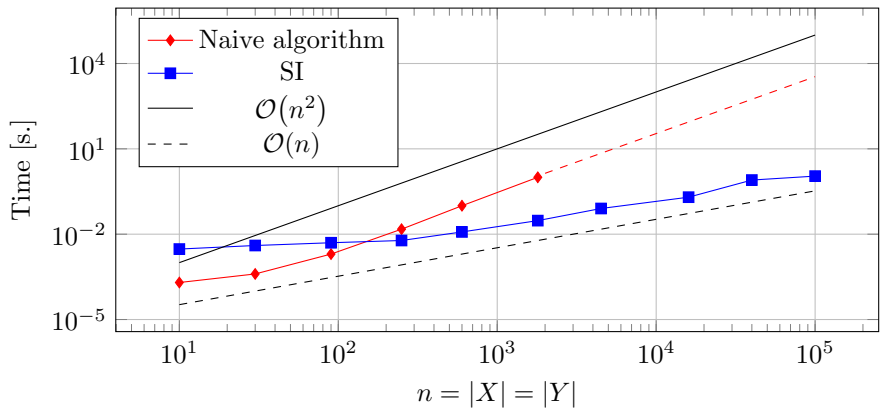
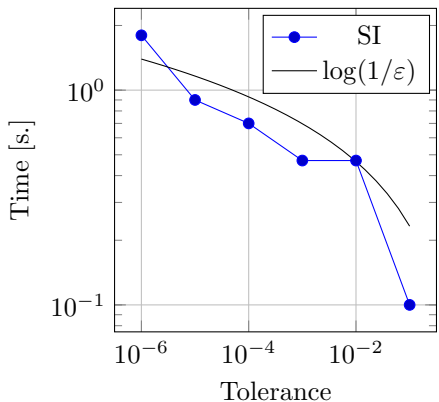
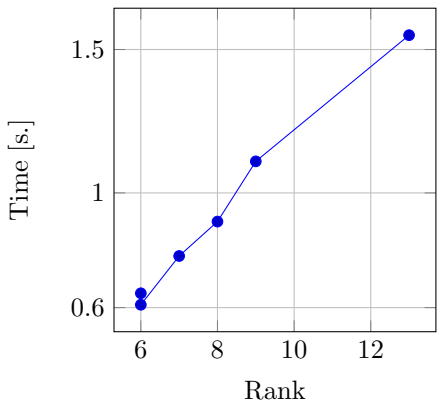

\section{Conclusion}

In this work, we built a kernel matrix low-rank approximation based on Skeletonized interpolation. This can be seen as an optimal way to interpolate families of functions using a custom basis.

This type of interpolation, by design, is always at least as good as polynomial interpolation as it always requires the minimal number of basis functions for a given approximation error.
We proved in this paper the asymptotic convergences of the scheme for kernels exhibiting fast (i.e., faster than polynomial) decay of singular values. We also proved the numerical stability of general Schur-complement types of formulas when using a backward stable algorithm.

In practice, the algorithm exhibits a low computational complexity of $\OO{n r}$ with small constants and is very simple to use. 
Furthermore, the accuracy can be set a priori and in practice, we observe nearly optimal convergence of the algorithm. 
Finally, the algorithm is completely insensible to the mesh point distribution, leading to more stable sets of ``pivots'' than Random Sampling or ACA.

\section*{Acknowledgements}
We would like to thank Cleve Ashcraft for his ideas and comments on the paper, as well as the anonymous reviewer for his careful reading and pertinent suggestions that greatly improved the paper.

\appendix

\section{Proofs of the theorems}

\begin{proof}[\autoref{lemma:1}]
This bound on the Lagrange basis is a classical result related to the growth of the Lebesgue constant in polynomial interpolation. For $\nx$ Chebyshev nodes of the first kind on $[-1,1]$ and the associated Lagrange basis functions $\ell_1, \dots, \ell_\nx$ we have the following result \cite[equation 13]{ibrahimoglu2016lebesgue}
        \[ \max_{x \in [-1,1]} \sum_{i=1}^\nx |\ell_i(x)| \leq \frac 2\pi \log(\nx+1) + 0.974 = \OO{\log(\nx)} \]
        This implies that in one dimension,
        \[ \| S(x, \Xcheb) \|_2 \leq \| S(x, \Xcheb) \|_1 = \OO{\log \nx} \]
        Going from one to $d$ dimensions can be done using Kronecker products. Indeed, for $x \in \R^d$,
        \[ S(x, \Xcheb) = S(x_1, \Xcheb_1) \otimes \cdots \otimes S(x_d, \Xcheb_d) \]
        where $x = (x_1, \dots, x_d)$ and $\Xcheb_1, \dots, \Xcheb_d$ are the one-dimensional Chebyshev nodes.
        Since for all $a \in \R^m, b \in \R^n$, $\|a \otimes b\|_2 = \sqrt{\sum_{i,j} (a_i b_j)^2} = \|a b^\top\|_F = \|a\|_2 \|b\|_2$, it follows that
        \[ \| S(x, \Xcheb) \|_2 = \| S(x_1, \Xcheb_1) \otimes \cdots \otimes S(x_d, \Xcheb_d) \|_2 = \prod_{i=1}^d \| S(x_i, \Xcheb_i) \|_2 = \prod_{i=1}^d \OO{\log \nx_i} \]
        This implies, using a fairly loose bound,
        \[ \| S(x, \Xcheb) \|_2 = \OO{ \log(|\Xcheb|)^d} \]
        The same argument can be done for $T(y, \Ycheb)$.

In 1D, the weights are
    \[ w_k = \frac{\pi}{\nx} \sin \left( \frac{2k-1}{2\nx} \pi \right) \]
        for $k = 1, \dots, \nx$. Obviously, $w_k > 0$.
        Clearly, $w_k < \frac{\pi}{\nx}$. Also, the minimum being reached at $k = 1$ or $k = \nx$,
        \[ w_k \geq \frac \pi \nx \sin\left(\frac{\pi}{2\nx}\right) > \frac \pi \nx \frac{2 \pi}{2 \pi \nx} = \frac{\pi}{\nx^2} \]
        Since the nodes in $d$ dimensions are products of the nodes in 1D, it follows that
\begin{gather*}
\| \DWx \|_2 \leq \frac{\pi^d}{\nx} \\
\| \DWx^{-1} \|_2 \leq \frac{\nx^2}{\pi^d}
\end{gather*}
The result follows.
\end{proof}

    \begin{proof}[\autoref{lemma:4}]
        We show the result for the second equation. This requires using, consecutively, the interpolation result and the CUR decomposition one. First, one can write from \autoref{lemma:1} and the interpolation,
        \[ \Kfun(\Xcheba, \Ycheba)^{-1} \Kfun(\Xcheba, y) = \Kfun(\Xcheba, \Ycheba)^{-1} \left[ \Kfun(\Xcheba, \Ycheb) T(y, \Ycheb)^\top + \Eint(\Xcheba, y) \right] \]
        Then, introducing the weight matrices and applying \autoref{lemma:3} on the interpolation matrix,
            \begin{align*} & \Kfun(\Xcheba, \Ycheb) = \\
                & = \DWxa^{-1/2} \DWxa^{1/2} \Kfun(\Xcheba, \Ycheb) \DWy^{1/2} \DWy^{-1/2} \\
                & = \DWxa^{-1/2} \Big\{ \DWxa^{1/2} \Kfun(\Xcheba, \Ycheba) \DWya^{1/2} \begin{bmatrix} I & \widecheck T^\top \end{bmatrix} \\
                & \hspace{5em} + \Eqr(\Xcheba, \Ycheb) \Big\} \DWy^{-1/2}
                \end{align*}
        Finally, combining and distributing all the factors gives us
                \begin{align*} \Kfun(\Xcheba, \Ycheba)^{-1} \Kfun(\Xcheba, y)
                    & = \DWya^{1/2}
                \begin{bmatrix} I
                & \widecheck T^\top
                \end{bmatrix}
                \DWy^{-1/2} T(y, \Ycheb)^\top \\
                & \hspace{-3em} + \Kfun(\Xcheba, \Ycheba)^{-1} \DWxa^{-1/2} \Eqr(\Xcheba, \Ycheb) \DWy^{-1/2} T(y, \Ycheb)^\top \\
            & \hspace{-3em} + \Kfun(\Xcheba, \Ycheba)^{-1} \Eint(\Xcheba, y)
            \end{align*}
        Here, we can bound all terms:
        \begin{itemize}
            \item For the first term, \autoref{lemma:1} and \autoref{lemma:3} show that the expression is bounded by a polynomial;
            \item For the second term use the fact that
                \[ \| \Kwhat^{-1} \|_2 \leq p^2(r_0,r_1) \frac 1\epsilon \Rightarrow \| \Kfun(\Xcheba, \Ycheba)^{-1} \|_2 = p'(r_0,r_1) \frac 1\epsilon \]
                hence, since $\| E_{QR}(\Xcheb,\Ycheb) \|_2 = \epsilon$, the product is again bounded by a polynomial since the $\epsilon$ cancel out ;
            \item The last term can be bounded in a similar way using 
                \[ \Eint(x,y) = \OO{\delta} \revrev{ \leq } \OO{\epsilon} \]
        \end{itemize}
        We conclude that there exists a polynomial $q$ such that
        \[ \| \Kfun(\Xcheba, \Ycheba)^{-1} \Kfun(\Xcheba, y) \|_2 = \OO{q(r_0, r_1)} \]
        The proof is similar in $x$.
    \end{proof}

    \begin{proof}[\autoref{thm:ai}]
    Combining interpolation and CUR decomposition results one can write
        \begin{align*}
             & \Kfun(x, y) = S(x, \Xcheb) \Kfun(\Xcheb, \Ycheb) T(y, \Ycheb)^\top + \Eint(x,y) \\
             & = S_w(x, \Xcheb) \Kfun_w(\Xcheb, \Ycheb) T_w(y, \Ycheb)^\top + \Eint(x,y) \\
             & = S_w(x, \Xcheb) \left[ \begin{bmatrix} I \\ \widecheck S \end{bmatrix} \Kfun_w(\Xcheba, \Ycheba) \begin{bmatrix} I & \widecheck T^\top \end{bmatrix} + \Eqr(\Xcheb, \Ycheb) \right] T_w(y, \Ycheb)^\top + \Eint(x,y) \\
             & = S_w(x, \Xcheb) \begin{bmatrix} \Kfun_w(\Xcheba, \Ycheba) \\ \widecheck S \Kfun_w(\Xcheba, \Ycheba) \end{bmatrix} \Kfun_w(\Xcheba, \Ycheba)^{-1} \begin{bmatrix} \Kfun_w(\Xcheba, \Ycheba) & \Kfun_w(\Xcheba, \Ycheba) \widecheck T^\top \end{bmatrix} T_w(y, \Ycheb)^\top \\
             & \quad + S_w(x, \Xcheb) \Eqr(\Xcheb, \Ycheb) T_w(y, \Ycheb)^\top + \Eint(x,y) \\
             & = S_w(x, \Xcheb) \left[ \Kfun_w(\Xcheb, \Ycheba) + \Eqr(\Xcheb, \Ycheba) \right]
             \Kfun_w(\Xcheba, \Ycheba)^{-1} \left[ \Kfun_w(\Xcheba, \Ycheb) + \Eqr(\Xcheba, \Ycheb) \right] \\
             & \quad \times T_w(y, \Ycheb)^\top
                 + S_w(x, \Xcheb) \Eqr(\Xcheb, \Ycheb) T_w(y, \Ycheb)^\top + \Eint(x,y) \\
             & = (\Kfun(x, \Ycheba) + \Eint(x, \Ycheba)) \Kfun(\Xcheba, \Ycheba)^{-1} (\Kfun(\Xcheba, y) + \Eint(\Xcheba, y)) \\
                 & \quad + S_w(x, \Xcheb) \Eqr(\Xcheb, \Ycheba) \Kfun_w(\Xcheba, \Ycheba)^{-1} \Kfun_w(\Xcheba, \Ycheb) T_w(y, \Ycheb)^\top \\
                 & \quad + S_w(x, \Xcheb) \Kfun_w(\Xcheb, \Ycheba) \Kfun_w(\Xcheba, \Ycheba)^{-1} \Eqr(\Xcheba, \Ycheb) T_w(y, \Ycheb)^\top \\
                 & \quad + S_w(x, \Xcheb) \Eqr(\Xcheb, \Ycheba) \Kfun_w(\Xcheba, \Ycheba)^{-1} \Eqr(\Xcheba, \Ycheb) T_w(y, \Ycheb)^\top \\
             & \quad + S_w(x, \Xcheb) \Eqr(\Xcheb, \Ycheb) T_w(y, \Ycheb)^\top + \Eint(x,y)
        \end{align*}
Distributing everything, factoring the weights matrices and simplifying, we obtain the following, where we indicate the bounds on each term on the right,
        \begin{align*}
& \Kfun(x, y) = \Kfun(x, \Ycheba) \Kfun(\Xcheba, \Ycheba)^{-1} \Kfun(\Xcheba, y) & \text{Approximation} \\
& \quad + \Eint(x, \Ycheba) \Kfun(\Xcheba, \Ycheba)^{-1} \Kfun(\Xcheba, y) & \OO{\delta q(r_0, r_1)} \\
& \quad + \Kfun(x, \Ycheba) \Kfun(\Xcheba, \Ycheba)^{-1} \Eint(\Xcheba, y) & \OO{\delta q(r_0, r_1)} \\
& \quad + \Eint(x, \Ycheba) \Kfun(\Xcheba, \Ycheba)^{-1} \Eint(\Xcheba, y) & \OO{\delta p'(r_0, r_1)} \\
& \quad + S(x, \Xcheb) \DWx^{-1/2} \Eqr(\Xcheb, \Ycheba) \DWya^{-1/2} & \\
& \qquad \Kfun(\Xcheba, \Ycheba)^{-1} \Kfun(\Xcheba, \Ycheb) T(y, \Ycheb)^\top & \OO{\epsilon (\log r_0)^{2d} q(r_0, r_1) r_0^2} \\
& \quad + S(x, \Xcheb) \Kfun(\Xcheb, \Ycheba) \Kfun(\Xcheba, \Ycheba)^{-1} \DWxa^{-1/2} & \\
& \qquad \Eqr(\Xcheba, \Ycheb) \DWy^{-1/2} T(y, \Ycheb)^\top & \OO{\epsilon (\log r_0)^{2d} q(r_0, r_1) r_0^2}\\
& \quad + S(x, \Xcheb) \DWx^{-1/2} \Eqr(\Xcheb, \Ycheba) & \\
& \qquad \DWya^{-1/2} \Kfun(\Xcheba, \Ycheba)^{-1} \DWxa^{-1/2} & \\
& \qquad \Eqr(\Xcheba, \Ycheb) \DWy^{-1/2} T(y, \Ycheb)^\top & \OO{ \epsilon (\log r_0)^{2d} r_0^2 p(r_0,r_1) }\\
& \quad + S(x, \Xcheb) \DWx^{-1/2} \Eqr(\Xcheb, \Ycheb) & \\
& \qquad \DWy^{-1/2} T(y, \Ycheb)^\top & \OO{\epsilon (\log r_0)^{2d} r_0^2} \\
& \quad + \Eint(x,y) & \OO{\delta}
\end{align*}
This concludes the proof.
\end{proof}

\bibliography{biblio}
\bibliographystyle{siamplain}

\end{document}